\documentclass[aop,seceqn,MSNbibl,citesort,dvips]{arximspdf}

\doi{10.1214/10-AOP616}
\volume{40}
\issue{1}
\pubyear{2012}
\firstpage{74}
\lastpage{102}

\makeatletter

\newcommand{\eqref}[1]{(\ref{#1})}

\newcommand{\Prb}{\mathbf{P}}
\newcommand{\Mean}{\mathbf{E}}

\newcommand{\Law}{\operatorname{Law}}
\newcommand{\trace}{\operatorname{tr}}

\newtheorem{thrm}{Theorem}[section]
\newtheorem{lemma}[thrm]{Lemma}

\newproclaim{rmk}[thrm]{Remark}
\newproclaim{defn}{Definition}

\makeatother

\begin{document}
\begin{frontmatter}

\title{Large deviation properties of weakly interacting processes via
weak convergence methods}
\runtitle{Large deviations for weakly interacting processes}

\begin{aug}
\author[A]{\fnms{Amarjit} \snm{Budhiraja}\thanksref{t2}\ead[label=e1]{budhiraj@email.unc.edu}},
\author[B]{\fnms{Paul} \snm{Dupuis}\thanksref{t3}\ead[label=e2]{dupuis@dam.brown.edu}} and
\author[B]{\fnms{Markus} \snm{Fischer}\corref{}\thanksref{t4}\ead[label=e3]{fischer@statlab.uni-heidelberg.de}}
\runauthor{A. Budhiraja, P. Dupuis and M. Fischer}
\affiliation{University of North Carolina at Chapel Hill, Brown University and~Brown~University}
\address[A]{A. Budhiraja\\
Department of Statistics\\
\quad and Operations Research\\
University of North Carolina\\
\quad at Chapel Hill\\
Chapel Hill, North Carolina 27599\\
USA\\
\printead{e1}}
\address[B]{P. Dupuis\\
M. Fischer\\
Lefschetz Center for Dynamical Systems\\
Division of Applied Mathematics\\
Brown University\\
Providence, Rhode Island 02912\\
USA\\
\printead{e2}\\
\phantom{E-mail: }\printead*{e3}}
\end{aug}

\thankstext{t2}{Supported in part by the Army Research Office
(W911NF-0-1-0080, W911NF-10-1-0158), NSF Grant DMS-10-04418 and
the US-Israel Binational Science Foundation 2008466.}
\thankstext{t3}{Supported in part by NSF Grant DMS-07-06003,
the Army Research Office W911NF-09-1-0155 and the Air Force Office of
Scientific Research FA9550-09-1-0378.}
\thankstext{t4}{Supported by the German Research
Foundation (DFG research fellowship), NSF Grant DMS-07-06003 and the
Air Force Office of Scientific Research (FA9550-07-1-0544, FA9550-09-1-0378).}

\received{\smonth{12} \syear{2009}}
\revised{\smonth{10} \syear{2010}}

%
\begin{abstract}
We study large deviation properties of systems of weakly interacting
particles modeled by It{\^{o}} stochastic differential equations
(SDEs). It is known under certain conditions that the
corresponding sequence of empirical measures converges, as the number
of particles tends to infinity, to the weak solution of an associated
McKean--Vlasov equation. We derive a large deviation principle via the
weak convergence approach. The proof, which avoids discretization
arguments, is based on a representation theorem, weak convergence and
ideas from stochastic optimal control. The method works under rather
mild assumptions and also for models described by SDEs not of
diffusion type. To illustrate this, we treat the case of SDEs
with delay.
\end{abstract}

%

\end{frontmatter}

\section{Introduction}

Collections of weakly interacting random processes have long been of
interest in statistical physics and more recently have appeared in
problems of engineering and operations research. A simple but important
example of such a collection is a group of ``particles,'' each of which evolves according to the
solution of an It\^{o}-type stochastic differential equation
(SDE). All particles have the same functional form for the drift and
diffusion coefficients. The coefficients of particle $i$ are, as usual,
allowed to depend on the current state of particle $i$, but also depend
on the current empirical distribution of all particle locations. When
the number of particles is large the contribution of any given particle
to the empirical distribution is small, and in this sense the
interaction between any two particles is considered ``weak.''

For various reasons, including model simplification and approximation,
one may consider a functional law of large numbers (LLN) limit
as the number of particles tends to infinity. The limit behavior of a
single particle (under assumptions which guarantee that all particles
are in some sense exchangeable) can be described by a two component
Markov process. One component corresponds to the state of a typical
particle, while the second corresponds to the limit of the empirical
measures. Again using that all particles are exchangeable, under
appropriate conditions one can show that the second component coincides
with the distribution of the particle component. The limit process,
which typically has an infinite-dimensional state, is sometimes
referred to as a ``nonlinear
diffusion.'' Because the particle's own distribution
appears in the state dynamics, the partial differential equations that
characterize expected values and densities associated with this process
are nonlinear, and hence the terminology.

In this paper we consider the large deviation properties of the
particle system as the number of particles tends to infinity. Thus the
deviations we study are those of the empirical measure of the prelimit
process from the distribution of the nonlinear diffusion. Of particular
interest, and a subject for further study, are deviations when the
initial distribution of the single particle in the nonlinear diffusion
is invariant under the joint particle/measure dynamics, and related
questions of stability for both the limit and prelimit processes.

One of the basic references for large deviation results for weakly
interacting diffusions is \cite{dawsongaertner87}. This paper considers
a system of uniformly nondegenerate diffusions with interaction in the
drift term and establishes a large deviation principle for the
empirical measure using discretization arguments and careful
exponential probability estimates (see Section \ref
{SectExtensionsComparison}). Properties related to a large deviation
principle such as fluctuation theorems have been studied in
\cite{tanaka82,benarousbrunaud90,meleard95,benarouszeitouni99,herrmannetalii08}.
A proof of the large deviation principle for systems with constant
diffusion coefficient that is based on a comparison result for a
related infinite-dimensional Hamilton--Jacobi--Bellman equation appears
in \cite{fengkurtz06}, Section 13.3.

Later works have developed the theory for a variety of alternative
models, including multilevel large deviations \cite{dawsongaertner94,djehicheschied98},
jump diffusions \cite{leonard95,leonard95b}, discrete-time systems
\cite{dawsondelmoral05,delmoralguionnet98} and interacting diffusions with random interaction
coefficients \cite{benarousguionnet95} or singular interaction
\cite{fontbona04}. In the current work we develop an approach which is very
different from the one taken in any of these papers. Our proofs do not
involve any time or space discretization of the system, and no
exponential probability estimates are invoked. The main ingredients in
the proof are weak convergence methods for functional occupation
measures and certain variational representation formulas. Our proofs
cover models with degenerate noise and allow for interaction in both
drift and diffusion terms. In fact, the techniques are applicable to a
wide range of model settings, and an example of stochastic delay
equations is considered in Section \ref{SectExtensions} to illustrate
the possibilities.

The starting point of our analysis is a variational representation for
moments of nonnegative functionals of a Brownian motion \cite
{bouedupuis98}. Using this representation, the proof of the large
deviation principle reduces to the study of asymptotic properties of
certain controlled versions of the original process. The key step in
the proof is to characterize the weak limits of the control and
controlled process as the large deviation parameter tends to its limit
and under the same scaling that applies to the original process. More
precisely, one needs to characterize the limit of the empirical measure
of a large collection of controlled and weakly interacting processes.
In the absence of control this characterization problem reduces to an
LLN analysis of the original particle system, which has been
studied extensively \cite{oelschlaeger84,funaki84,gaertner88}. Our
main tools for the study of the controlled analog are functional
occupation measure methods. Indeed, these methods have been found to be
quite useful for the study of averaging problems, but where the average
is with respect to a time variable \cite{kushner90}. In the problem
studied here the measure-valued processes of interest are obtained
using averaging over particles rather than the time variable.

The approach presented here can be applied to interacting systems
driven by general continuous time processes with jumps provided the
systems are scaled in the right way. Indeed, the driving noise process
could be a Brownian motion plus an independent Poisson random measure.
A key step to make the approach work is a variational representation of
Poisson functionals, which has recently been established in \cite{budhirajaetal09}.

Finally, we remark that variational representations for Brownian
motions and Poisson random measures
\cite{budhirajadupuis00,budhirajaetal08,budhirajaetal09} have proved to be useful for the
study of small-noise large-deviation problems, and many recent papers
have applied these results to a variety of infinite-dimensional
small-noise systems. A small selection is
\cite{DuanMillet09,renzhang05,roeckneretal09,sritharansundar06} (see
\cite{budhirajaetal09} for a more complete list). We expect the current work
to be similarly a starting point for the study, using variational
representations, of a rather different collection of large deviation
problems, namely asymptotics of a large number of interacting particles.

An outline of the paper is as follows. In Section \ref{SectModel} we
introduce the interacting SDE particle model, the related
controlled and LLN limit versions and discuss the relevant
topologies and sense of uniqueness of solutions. Section \ref
{SectLPLimit} discusses the relation between Laplace and
large-deviation principles, states assumptions and the main result of
the paper and then outlines how this result will be proved using a
representation theorem. In Section \ref{SectAuxiliary} we describe the
martingale problems that will be used in the proof. The proof itself is
divided into lower and upper bounds in Sections \ref{SectLPLowerBound}
and \ref{SectLPUpperBound}, respectively. The constructions in the
proof are set up to handle a more general case than just the model
introduced in Section \ref{SectModel}, and in Section \ref
{SectExtensions} we use this generality to state and prove a large
deviation theorem for systems with delay. This section also reviews the
prior work of \cite{dawsongaertner87}. The \hyperref[AppLocalMartingales]{Appendix} contains the proof
of a technical point that was deferred for reasons of exposition.


\section{The model}
\label{SectModel}

For each $N\in\mathbb{N}$, the $N$-particle prelimit model is
described in terms of a system of $N$ weakly coupled $d$-dimensional
stochastic differential equations (SDEs). The system is
considered over the fixed finite interval $[0,T]$. Set $\mathcal{X}
\doteq\mathbf{C}([0,T],\mathbb{R}^{d})$, and equip $\mathcal{X}$ with
the maximum norm, which is denoted by $\Vert\cdot\Vert$. Similarly, set
$\mathcal{W} \doteq\mathbf{C}([0,T],\mathbb{R}^{d_{1}})$ and equip
$\mathcal{W}$ with the maximum norm. Let $(\Omega,\mathcal{F},\Prb)$
be a probability space, and suppose that on this space there is a
filtration $(\mathcal{F}_{t})$ satisfying the usual conditions [i.e.,
$(\mathcal{F}_{t})$ is right-continuous and $\mathcal{F}_{0}$ contains
all $\Prb$-negligible sets], as well as a~collection $\{W^{i},i\in
\mathbb{N}\} $ of independent standard $d_{1}$-dimensional
$(\mathcal{F}_{t})$-Wiener processes.

Let $b$ and $\sigma$ be Borel measurable functions defined on $\mathbb
{R}^{d}\times\mathcal{P}(\mathbb{R}^{d})$ taking values in $\mathbb
{R}^{d}$ and the space of real $d \times d_{1}$-matrices,
respectively. If $(\mathcal{S},d_{\mathcal{S}})$ is a metric space,
then $\mathcal{P}(\mathcal{S})$ denotes the space of probability
measures on the Borel $\sigma$-field $\mathcal{B}(\mathcal{S})$. The
space $\mathcal{P}(\mathcal{S})$ is equipped with the topology of weak
convergence, which can be metricized, using, for example, the bounded
Lipschitz metric, making it a Polish space.

The evolution of the state of the particles in the $N$-particle model
is given by the solution to the system of \mbox{SDEs}
\begin{eqnarray} \label{EqPrelimitSDE}
dX^{i,N}(t)&=&b(X^{i,N}(t),\mu^{N}(t))\,dt+\sigma
(X^{i,N}(t),\mu^{N}(t))\,dW^{i}(t), \nonumber\\[-8pt]\\[-8pt]
X^{i,N}(0)&=&x^{i,N},\nonumber
\end{eqnarray}
where $x^{i,N} \in\mathbb{R}^{d}$, $i\in\{1,\ldots,N\}$, and
\[
\mu^{N}(t,\omega)\doteq\frac{1}{N}\sum_{i=1}^{N}\delta
_{X^{i,N}(t,\omega)},\qquad \omega\in\Omega,
\]
is the empirical measure of $(X^{1,N}(t),\ldots,X^{N,N}(t))$ for $t\in
[0,T]$. By construction, $\mu^{N}(t)$ is a $\mathcal{P}(\mathbb
{R}^{d})$-valued random variable. Denote by $\mu^{N}$ the empirical
measure of $(X^{1,N},\ldots,X^{N,N})$ over the time interval $[0,T]$,
that is, $\mu^{N}$ is the $\mathcal{P}(\mathcal{X})$-valued random
variable defined by
\[
\mu_{\omega}^{N}\doteq\frac{1}{N}\sum_{i=1}^{N}\delta_{X^{i,N}(\cdot,\omega
)},\qquad \omega\in\Omega.
\]
Clearly, the distribution of $\mu^{N}(t)$ is identical to the marginal
distribution of~$\mu^{N}$ at time $t$, that is, $\mu^{N}(t) = \mu
^{N}\circ\pi^{-1}_{t}$ where $\pi_{t} \dvtx \mathcal{X} \rightarrow\mathbb
{R}^{d}$ is the projection map corresponding to the value at time
$t$.

Our aim is to establish a Laplace principle for the family $\{\mu^{N},
N\in\mathbb{N}\}$ of $\mathcal{P}(\mathcal{X})$-valued random
variables. When $\frac{1}{N}\sum_{i=0}^{N}\delta_{x^{i,N}}$ converges
weakly to $\nu_{0}$ for some $\nu_{0} \in\mathcal{P}(\mathbb{R}^{d})$,
the asymptotic behavior of $\mu^{N}$ as $N$ tends to infinity can be
characterized in terms of solutions to the nonlinear diffusion
\begin{eqnarray} \label{EqLimitSDE}
dX(t)&=&b(X(t),\Law(X(t)))\,dt+\sigma(X(t),\Law(X(t))
)\,dW(t),\nonumber\\[-8pt]\\[-8pt]
X(0) &\sim&\nu_{0},\nonumber
\end{eqnarray}
where $W$ is a standard $d_{1}$-dimensional Wiener process. Thus we are
interested in the study of deviations of $\mu^{N}$, $N$ large, from its
typical behavior, namely the probability law of the process solving
\eqref{EqLimitSDE}.

In the formulation and proof of the Laplace principle, we will need to
consider a controlled version of \eqref{EqPrelimitSDE}. For
$N\in\mathbb{N}$, let $\mathcal{U}_{N}$ be the space of all $(\mathcal
{F}_{t})$-progressively measurable functions $u \dvtx[0,T]\times\Omega
\rightarrow\mathbb{R}^{N\times d_{1}}$ such that
\[
\Mean\biggl[ \int_{0}^{T}|u(t)|^{2}\,dt\biggr] <\infty,
\]
where $\Mean$ denotes expectation with respect to $\Prb$, and \mbox{$|\cdot|$}
denotes the Euclidean norm of appropriate dimension. For $u\in\mathcal
{U}_{N}$, we sometimes write $u=(u_{1},\ldots,u_{N})$, where $u_{i}$
is the $i$th block of $d_{1}$ components of $u$.

Given $u\in\mathcal{U}_{N}$, $u=(u_{1},\ldots,u_{N})$, we consider
the controlled system of \mbox{SDEs}
\begin{eqnarray} \label{EqPrelimitControlSDE}
d\bar{X}^{i,N}(t)&=&b(\bar{X}^{i,N}(t),\bar{\mu}^{N}(t))\,dt
+\sigma(\bar{X}^{i,N}(t),\bar{\mu}^{N}(t))u_{i}(t)\,dt
\nonumber\\[-8pt]\\[-8pt]
&&{}  +\sigma(\bar{X}^{i,N}(t),\bar{\mu}^{N}(t)
)\,dW^{i}(t),\qquad \bar{X}^{i,N}(0)=x^{i,N},
\nonumber
\end{eqnarray}
where $\bar{\mu}^{N}(t)$ and $\bar{\mu}^{N}$ are the empirical measures
of $\bar{X}^{i,N}(t)$ and $\bar{X}^{i,N}$, respectively,
\[
\bar{\mu}^{N}(t,\omega)\doteq\frac{1}{N}\sum_{i=1}^{N}\delta_{\bar
{X}^{i,N}(t,\omega)}, \qquad\bar{\mu}_{\omega}^{N}\doteq\frac
{1}{N}\sum_{i=1}^{N}\delta_{\bar{X}^{i,N}(\cdot,\omega)}, \qquad\omega\in
\Omega.
\]
The ``barred'' symbols in the display
above and in \eqref{EqPrelimitControlSDE} refer to objects
depending on a control, here $u$. We adopt this as a convention and
indicate control-dependent objects by overbars. The existence and
uniqueness of strong solutions to \eqref{EqPrelimitControlSDE}
will be a consequence of assumption \hyperlink{APrelimitSolutions}{(A3)} made in
Section \ref{SectLPLimit}; see comments below assumption
\hyperlink{ATightness}{(A5)} there.

It will be convenient to have a path space which is Polish for the
components $u_{i}$, $i\in\{1,\ldots,N\}$, of a control process $u \in
\mathcal{U}_{N}$. We choose the space of deterministic relaxed controls
on $\mathbb{R}^{d_{1}}\times[0,T]$ with finite first moments. Let us
first recall some facts about deterministic relaxed controls
(see, e.g.,~\cite{kushner90}, Section 3.2, for the case of a compact space
of control actions). Denote by $\mathcal{R}$ the space of all
deterministic relaxed controls on $\mathbb{R}^{d_{1}}\times[0,T]$,
that is, $\mathcal{R}$ is the set of all positive measures $r$ on
$\mathcal{B}(\mathbb{R}^{d_{1}} \times [0,T])$ such that $r(\mathbb
{R}^{d_{1}} \times [0,t])=t$ for all $t\in[0,T] $. If $r\in\mathcal
{R}$ and $B\in\mathcal{B}(\mathbb{R}^{d_{1}})$, then the mapping
$[0,T]\ni t\mapsto r(B \times [0,t])$ is absolutely continuous, hence
differentiable almost everywhere. Since $\mathcal{B}(\mathbb
{R}^{d_{1}})$ is countably generated, the time derivative of $r$ exists
almost everywhere and is a measurable mapping $r_{t} \dvtx
[0,T]\rightarrow\mathcal{P}(\mathbb{R}^{d_{1}})$ such that $r(dy
\times dt) = r_{t}(dy)\,dt$.

Denote by $\mathcal{R}_{1}$ the space of deterministic relaxed controls
with finite first moments, that is,
\[
\mathcal{R}_{1}\doteq\biggl\{ r\in\mathcal{R}\dvtx\int_{\mathbb
{R}^{d_{1}}\times[0,T]}|y| r(dy \times dt)<\infty\biggr\}.
\]
By definition, $\mathcal{R}_{1}\subset\mathcal{R}$. The topology of
weak convergence of measures turns~$\mathcal{R}$ into a Polish space
(not compact in our case). We equip $\mathcal{R}_{1}$ with the topology
of weak convergence of measures plus convergence of first moments. This
topology turns $\mathcal{R}_{1}$ into a Polish space (cf.
\cite{rachev91}, Section 6.3). It is related to the Monge--Kantorovich
distances. For $T=1$ (else one has to renormalize), the topology
coincides with that induced by the Monge--Kantorovich distance with
exponent one, also called the Kantorovich--Rubinstein distance or
Wasserstein distance of order one. The topology is convenient because
the controls appear in an unbounded (but affine) fashion in the
dynamics. Thus ordinary weak convergence will not imply convergence of
corresponding integrals, but convergence in $\mathcal{R}_{1}$ will.

Any $\mathbb{R}^{d_{1}}$-valued process $v$ defined on some probability
space $(\tilde{\Omega},\tilde{\mathcal{F}},\tilde{\Prb})$ induces an
$\mathcal{R}$-valued random variable $\rho$ according to
\begin{eqnarray} \label{ExControlRelaxation}
\rho_{\omega}(B\times I)\doteq\int_{I}\delta_{v(t,\omega
)}(B)\,dt, \nonumber\\[-8pt]\\[-8pt]
&&\eqntext{B\in\mathcal{B}(\mathbb{R}^{d_{1}}),
I\subset[0,T],
\omega\in\tilde{\Omega}.}
\end{eqnarray}
If $v$ is such that $\int_{0}^{T} |v(t,\omega)|\,dt < \infty$ for all
$\omega\in\tilde{\Omega}$, then the induced random variable $\rho$
takes values in $\mathcal{R}_{1}$. If $v$ is progressively measurable
with respect to a filtration $(\tilde{\mathcal{F}}_{t})$ in $\tilde
{\mathcal{F}}$, then $\rho$ is adapted in the sense that the mapping
$t\mapsto\rho(B \times [0,t])$ is $(\tilde{\mathcal
{F}}_{t})$-adapted for all $B\in\mathcal{B}(\mathbb{R}^{d_{1}})$
\cite{kushner90}, Section 3.3.

Given an adapted (in the above sense) $\mathcal{R}_{1}$-valued random
variable $\rho$ and a~Borel measurable mapping $\nu \dvtx [0,T]\rightarrow
\mathcal{P}(\mathbb{R}^{d})$, we will consider the controlled \mbox{SDE}
\begin{eqnarray} \label{EqParamLimitControlSDE}
d\bar{X}(t) &=& b(\bar{X}(t),\nu(t))\,dt  + \biggl( \int
_{\mathbb{R}^{d_{1}}}\sigma(\bar{X}(t),\nu(t))y \rho
_{t}(dy)\biggr)\,dt \nonumber\\[-8pt]\\[-8pt]
&&{} + \sigma(\bar{X}(t),\nu(t))\,dW(t),\qquad \bar{X}(0)
\sim\nu(0),
\nonumber
\end{eqnarray}
where $W$ is a $d_{1}$-dimensional $(\tilde{\mathcal{F}}_{t})$-adapted
standard Wiener process. Equation \eqref{EqParamLimitControlSDE} is a
parameterized version of \eqref{EqLimitControlSDE} below, the
controlled analog of the limit \mbox{SDE} \eqref{EqLimitSDE}. We will
only have to deal with weak solutions of \eqref
{EqParamLimitControlSDE} or, equivalently, with certain probability
measures on $\mathcal{B}(\mathcal{Z})$, where
\[
\mathcal{Z}\doteq\mathcal{X}\times\mathcal{R}_{1}\times\mathcal{W}.
\]
For a typical element in $\mathcal{Z}$ let us write $(\varphi,r,w)$ with
the understanding that $\varphi\in\mathcal{X}$, $r\in\mathcal{R}_{1}$,
$w\in\mathcal{W}$.

Notice that we include $\mathcal{W}$ as a component of our canonical
space $\mathcal{Z}$. This will allow identification of the\vspace*{1pt} joint
distribution of the control and driving Wiener process. Indeed, if the
triple $(\bar{X},\rho,W)$ defined on some filtered probability space
$(\tilde{\Omega},\tilde{\mathcal{F}},\tilde{\Prb},(\tilde{\mathcal
{F}}_{t}))$ solves \eqref{EqParamLimitControlSDE} for some
measurable $\nu \dvtx [0,T]\rightarrow\mathcal{P}(\mathbb{R}^{d})$, then
the distribution of $(\bar{X},\rho,W)$ under $\tilde{\Prb}$ is an
element of $\mathcal{P}(\mathcal{Z})$.

When \eqref{EqParamLimitControlSDE} is used the mapping $\nu
\dvtx[0,T]\rightarrow\mathcal{P}(\mathbb{R}^{d})$ appearing in the
coefficients will be determined by a probability measure on $\mathcal
{B}(\mathcal{Z})$. To be more precise, let $\Theta\in\mathcal
{P}(\mathcal{Z})$. Then $\Theta$ induces a mapping $\nu_{\Theta}
\dvtx[0,T]\rightarrow\mathcal{P}(\mathbb{R}^{d})$ which is defined by
\begin{equation} \label{ExQMarginal}
\nu_{\Theta}(t)(B)\doteq\Theta\bigl(\{(\varphi,r,w)\in\mathcal{Z}\dvtx\varphi
(t)\in B\}\bigr),\qquad  B\in\mathcal{B}(\mathbb{R}^{d}),
t\in[0,T].\hspace*{-12pt}
\end{equation}
By construction, $\nu_{\Theta}(t)$ is the distribution under $\Theta$
of the first component of the coordinate process on $\mathcal
{Z}=\mathcal{X}\times\mathcal{R}_{1}\times\mathcal{W}$ at time $t$.
Therefore, if $\Theta$ corresponds to a weak solution of \eqref
{EqParamLimitControlSDE} with $\nu=\nu_{\Theta}$, then $\Theta$ also
corresponds to a weak solution of the controlled limit \mbox{SDE}
\begin{eqnarray} \label{EqLimitControlSDE}
d\bar{X}(t) &=& b(\bar{X}(t),\Law(\bar{X}(t)))\,dt + \biggl(\int
_{\mathbb{R}^{d_{1}}}\sigma(\bar{X}(t),\Law(\bar{X}(t)))y\rho
_{t}(dy)\biggr)\,dt \hspace*{-26pt}\nonumber\\[-8pt]\\[-8pt]
&&{} +\sigma(\bar{X}(t),\Law(\bar{X}(t)))\,dW(t), \qquad\bar
{X}(0)\sim\nu_{\Theta}(0).
\nonumber
\end{eqnarray}
Here $W$ is a $d_{1}$-dimensional standard Wiener process defined on
some probability space $(\tilde{\Omega},\tilde{\mathcal{F}},\tilde{\Prb
})$ carrying a filtration $(\tilde{\mathcal{F}}_{t})$, and $\rho$ is
an $(\tilde{\mathcal{F}}_{t})$-adapted $\mathcal{R}_{1}$-valued random
variable such that $(\bar{X},\rho,W)$ has distribution $\Theta$ under~$\tilde{\Prb}$.
The process triple $(\bar{X},\rho,W)$ can be given
explicitly as the coordinate process on the probability space $(\mathcal
{Z},\mathcal{B}(\mathcal{Z}),\Theta)$ endowed with the canonical
filtration $(\mathcal{G}_{t})$ in $\mathcal{B}(\mathcal{Z})$. More
precisely, the processes $\bar{X}$, $\rho$, $W$ are defined on
$(\mathcal{Z},\mathcal{B}(\mathcal{Z}))$ by\looseness=-1
\begin{eqnarray*}
\bar{X}(t,(\varphi,r,w))&\doteq&\varphi(t),\qquad \rho(t,(\varphi,r,w))\doteq
r_{|\mathcal{B}(\mathbb{R}^{d_{1}}\times[0,t])},\\
W(t,(\varphi,r,w))&\doteq& w(t).
\end{eqnarray*}\looseness=0
Here we abuse notation and use $\rho(t,\cdot)$ to denote the restriction
of a measure defined on $\mathcal{B}(\mathbb{R}^{d_{1}}\times[0,T])$
to $\mathcal{B}(\mathbb{R}^{d_{1}}\times[0,t])$. The canonical
filtration is given by
\[
\mathcal{G}_{t}\doteq\sigma\bigl( (\bar{X}(s),\rho(s),W(s))\dvtx0\leq
s\leq t\bigr),\qquad  t\in[0,T].
\]
Notice that $\rho(s)$ takes values in the space of deterministic
relaxed controls on $\mathbb{R}^{d_{1}}\times[0,s]$ with finite
first moments.\vadjust{\goodbreak}

One of the assumptions we make below [assumption\vspace*{1pt} \hyperlink{ALimitSolution}{(A4)}
in Section \ref{SectLPLimit}] is the weak uniqueness
of solutions to \eqref{EqLimitControlSDE}. If $((\tilde{\Omega
},\tilde{\mathcal{F}},\tilde{\Prb}),(\tilde{\mathcal{F}}_{t}),(\bar
{X},\rho,W))$ is a~weak solution of \eqref{EqLimitControlSDE},
then $\tilde{\Prb}\circ(\bar{X},\rho,W)^{-1} \in\mathcal{P}(\mathcal
{Z})$. The property of weak uniqueness can therefore be formulated in
terms of probability measures on~$\mathcal{B}(\mathcal{Z})$.
\begin{defn} \label{DefUniqueness}
\textit{Weak uniqueness} is said to hold for \eqref
{EqLimitControlSDE} if whenever~$\Theta$, $\tilde{\Theta} \in\mathcal
{P}(\mathcal{Z})$ are such that $\Theta$, $\tilde{\Theta}$ both
correspond to weak solutions of \eqref{EqLimitControlSDE}, $\nu
_{\Theta}(0) = \nu_{\tilde{\Theta}}(0)$ and $\Theta_{|\mathcal
{B}(\mathcal{R}_{1}\times\mathcal{W})} = \tilde{\Theta}_{|\mathcal
{B}(\mathcal{R}_{1}\times\mathcal{W})}$, then $\Theta= \tilde{\Theta}$.
\end{defn}

Thus, weak uniqueness for \eqref{EqLimitControlSDE} means
that, given any initial distribution for the state process, the joint
distribution of control and driving Wiener process uniquely determines
the distribution of the solution triple.


\section{Laplace principle}
\label{SectLPLimit}

A function $I \dvtx \mathcal{P}(\mathcal{X})\rightarrow[0,\infty]$ is
called a \textit{rate function} if for each $M<\infty$ the set $\{\theta
\in\mathcal{P}(\mathcal{X})\dvtx I(\theta)\leq M\}$ is compact (some
authors call such functions \textit{good} rate functions). We say that a
Laplace principle holds for the family $\{\mu^{N}, N\in\mathbb{N}\}$
with rate function $I$ if for any bounded and continuous function
$F\dvtx\mathcal{P}(\mathcal{X})\rightarrow\mathbb{R}$,
\begin{equation} \label{EqLPLimit}
\lim_{N\rightarrow\infty} -\frac{1}{N} \log\Mean[ \exp
(-N\cdot F(\mu^{N}))] = \inf_{\theta\in\mathcal{P}(\mathcal
{X})} \{F(\theta) + I(\theta)\}.
\end{equation}
It is well known that in our setting the Laplace principle holds if and
only if $\{\mu^{N}, N\in\mathbb{N}\}$ satisfies a large deviation
principle with rate function $I$ \cite{dupuisellis97}, Section 1.2.

Let us make the following assumptions about the functions $b$, $\sigma$
and the family $\{x^{i,N}\} \subset\mathbb{R}^{d}$ of initial conditions:

\begin{longlist}[(A5)]
\item[(A1)]\hypertarget{AInitialCondition} For some $\nu_{0}\in\mathcal
{P}(\mathbb{R}^{d})$, $\frac{1}{N} \sum_{i=1}^{N}\delta
_{x^{i,N}}\rightarrow\nu_{0}$ as $N$ tends to infinity.

\item[(A2)]\hypertarget{ACoeffContinuity} The coefficients $b$, $\sigma$ are continuous.

\item[(A3)]\hypertarget{APrelimitSolutions} For all $N \in\mathbb{N}$, existence
and uniqueness of solutions holds in the strong sense for the system of
$N$ equations given by \eqref{EqPrelimitSDE}.

\item[(A4)]\hypertarget{ALimitSolution} Weak uniqueness of solutions holds for
\eqref{EqLimitControlSDE}.

\item[(A5)]\hypertarget{ATightness} If $u^{N} \in\mathcal{U}_{N}$, $N\in\mathbb
{N}$, are such that
\[
\sup_{N\in\mathbb{N}} \Mean\Biggl[ \frac{1}{N} \sum_{i=1}^{N}\int
_{0}^{T}|u_{i}^{N}(t)|^{2} \,dt \Biggr]  <\infty,
\]
then $\{\bar{\mu}^{N}, N\in\mathbb{N}\}$ is tight as a family of
$\mathcal{P}(\mathcal{X})$-valued random variables, where $\bar{\mu
}^{N}$ is the empirical measure of the solution to the system of
\eqref{EqPrelimitControlSDE} under $u^{N}$.
\end{longlist}

Assumption \hyperlink{AInitialCondition}{(A1)} is a sort of law of large numbers
for the deterministic initial conditions. The assumption is
necessary
for the convergence of the empirical measures $\mu^{N}$ associated
with the state process. The continuity assumption \hyperlink{ACoeffContinuity}{(A2)}\vadjust{\goodbreak}
implies that the coefficients $b$, $\sigma$ are
uniformly continuous and uniformly bounded on sets $B\times P$, where
$B\subset\mathbb{R}^{d}$ is bounded and $P\subset\mathcal{P}(\mathbb
{R}^{d})$ is compact.

Assumption \hyperlink{APrelimitSolutions}{(A3)} about strong existence and
uniqueness of solutions for the prelimit model will be needed to
justify a variational representation for the cumulant generating
functionals appearing in \eqref{EqLPLimit}; see \eqref
{EqVariationalPrelimit} below. Assumption \hyperlink{APrelimitSolutions}{(A3)}
and an application of Girsanov's theorem show that \eqref
{EqPrelimitControlSDE} has a unique strong solution whenever $\int
_{0}^{T}|u(t)|^{2}\,dt\leq M$ $\Prb$-almost surely for some $M\in
(0,\infty)$. In fact, there is a Borel measurable mapping
$h^{N}=(h_{1}^{N},\ldots,h_{N}^{N})$ with $h_{i}^{N} \dvtx\Omega
\rightarrow\mathcal{X}$, $i\in\{1,\ldots,N\}$, such that, for $\Prb
$-almost all $\omega\in\Omega$, the unique strong solution of \eqref
{EqPrelimitSDE} is given as
\[
X^{i,N}(\cdot,\omega) = h_{i}^{N}(W(\cdot,\omega)),
\]
and under the above integrability condition on $u$, the unique strong
solution of \eqref{EqPrelimitControlSDE} equals $\Prb$-almost surely
\[
\bar{X}^{i,N}(\cdot,\omega)=h_{i}^{N}\biggl( W(\cdot,\omega)+\int
_{0}^{\cdot}u(s,\omega)\,ds \biggr).
\]
By a localization argument one can now show that \eqref
{EqPrelimitControlSDE} in fact has a unique strong solution for all
$u\in\mathcal{U}_{N}$, which is once more given by the above relation.

Weak uniqueness as stipulated in \hyperlink{ALimitSolution}{(A4)} for the
controlled nonlinear diffusions given by \eqref
{EqLimitControlSDE} is meant in the sense of Definition \ref
{DefUniqueness}. It is typical that such weak uniqueness holds if it
holds for the uncontrolled system (\ref{EqLimitSDE}).

Grant assumption \hyperlink{AInitialCondition}{(A1)}. Then
assumptions \hyperlink{ACoeffContinuity}{(A2)}--\hyperlink{ATightness}{(A5)} are all
satisfied if~$b$, $\sigma$ are uniformly Lipschitz [with respect to
the bounded Lipschitz metric on $\mathcal{P}(\mathbb{R}^{d})$] or
locally Lipschitz satisfying a suitable coercivity condition. A~simple
example of such a condition on $b$, $\sigma$ would be that for some
constant $C > 0$, all $x\in\mathbb{R}^{d}$ and all $\nu\in\mathcal
{P}(\mathbb{R}^{d})$,
\[
2\langle b(x,\nu),x\rangle+\trace(\sigma{\sigma}^{\mathsf{T}}
)(x,\nu)\leq C( 1+|x|^{2}).
\]

The reason for assumption \hyperlink{ATightness}{(A5)} being stated as it is, is
that there are many different sets of conditions on the problem data
(i.e., $b$ and $\sigma$) and the initial conditions which imply
tightness of the empirical measures of the $\bar{X}^{i,N}$. For
instance, \hyperlink{ATightness}{(A5)} is automatically satisfied if the
coefficients are bounded. It also holds if $b$, $\sigma$ are Lipschitz
continuous. More general conditions can be formulated in terms of the
action of the infinitesimal generator associated with \eqref
{EqLimitControlSDE}, given in \eqref{ExLimitGenerator} below, on some
``Lyapunov function'' $\varphi \dvtx \mathbb
{R}^{d}\rightarrow\mathbb{R}$; also see Section \ref
{SectExtensionsComparison}.

For a probability measure $\Theta\in\mathcal{P}(\mathcal{Z})$,
recalling that $\mathcal{Z}=\mathcal{X}\times\mathcal{R}_{1}\times
\mathcal{W}$, let~$\Theta_{\mathcal{X}}$, $\Theta_{\mathcal{R}}$ denote
the first and second marginal, respectively. Let $\mathcal{P}_{\infty
}$ be the set of all probability measures $\Theta\in\mathcal
{P}(\mathcal{Z})$ such that:

\begin{longlist}
\item\vspace*{-6pt}
\[
\int_{\mathcal{R}_{1}}\int_{\mathbb{R}^{d_{1}}\times[0,T]} |y|^{2}
r(dy \times dt)\Theta_{\mathcal{R}}(dr) <\infty;
\]

\item$\Theta$ corresponds to a weak solution of \eqref
{EqLimitControlSDE};\vadjust{\goodbreak}

\item$\nu_{\Theta}(0) = \nu_{0}$, where $\nu_{0}\in\mathcal
{P}(\mathbb{R}^{d})$ is the initial distribution from assumption~\hyperlink{AInitialCondition}{(A1)}.
\end{longlist}

The main result of this paper is the following.
\begin{thrm} \label{ThLaplacePrinciple}
Suppose that assumptions \textup{\hyperlink{AInitialCondition}{(A1)}--\hyperlink{ATightness}{(A5)}} hold. Then the
family of empirical measures $\{\mu^{N},N\in\mathbb{N}\}$ satisfies
the Laplace principle with rate function
\[
I(\theta) = \inf_{\Theta\in\mathcal{P}_{\infty}\dvtx\Theta_{\mathcal
{X}}=\theta}  \frac{1}{2}\int_{\mathcal{R}} \int_{\mathbb
{R}^{d_{1}}\times[0,T]} |y|^{2}  r(dy \times dt)\Theta_{\mathcal{R}}(dr).
\]
\end{thrm}
\begin{rmk} \label{RmAltRate}
The above expression for the rate function $I$ is convenient for
proving the Laplace principle. An alternative and perhaps more familiar
form of the rate function is the following. By definition of $\mathcal
{P}_{\infty}$, and since the control appears linearly in the limit
dynamics, we can write
\[
I(\theta)=\inf_{\Theta\in\mathcal{P}_{\infty}\dvtx\Theta
_{\mathcal{X}}=\theta} \Mean_{\Theta}\biggl[ \frac{1}{2}\int
_{0}^{T}|u(t)|^{2}\,dt\biggr],
\]
where $\inf\varnothing\doteq\infty$ by convention, $u(t)=\int_{\mathbb
{R}^{d_{1}}}y\rho_{t}(dy)$, $(\bar{X},W,\rho)$ is the canonical
process on $(\mathcal{Z},\mathcal{B}(\mathcal{Z}))$, and $\Theta
$-almost surely $\bar{X}$ satisfies
\begin{equation} \label{EqLimitControlSDE2}
d\bar{X}(t) = b(\bar{X}(t),\theta(t))\,dt +\sigma(\bar
{X}(t),\theta(t))u(t)\,dt
+\sigma(\bar{X}(t),\theta(t))\,dW(t).\hspace*{-28pt}
\end{equation}
\end{rmk}

The proof of Theorem \ref{ThLaplacePrinciple} is based on a
representation for functionals of Brownian motion, a martingale
characterization of weak solutions of \eqref
{EqLimitControlSDE} and weak convergence arguments.

By assumption \hyperlink{APrelimitSolutions}{(A3)}, for each $N\in\mathbb{N}$,
the $N$-particle system of \eqref{EqPrelimitSDE} possesses a
unique strong solution for the given initial condition. By Theorem~3.6
in \cite{budhirajadupuis00}, for any $F\in\mathbf{C}_{b}(\mathcal{X})$
the prelimit expressions in \eqref{EqLPLimit} can be rewritten as
\begin{eqnarray} \label{EqVariationalPrelimit}
&&-\frac{1}{N}\log\Mean\bigl[ \exp\bigl(-N\cdot F(\mu^{N})\bigr)
\bigr] \nonumber\\[-8pt]\\[-8pt]
&&\qquad = \inf_{u^{N}\in\mathcal{U}_{N}}\Biggl\{ \frac{1}{2}\Mean
\Biggl[ \frac{1}{N}\sum_{i=1}^{N}\int_{0}^{T}|u_{i}^{N}(t)|^{2}\,dt\Biggr] +
\Mean[ F(\bar{\mu}^{N})] \Biggr\},
\nonumber
\end{eqnarray}
where $\bar{\mu}^{N}$ is the empirical measure of the solution to the
system of \eqref{EqPrelimitControlSDE} under
$u^{N}=(u_{1}^{N},\ldots,u_{N}^{N})\in\mathcal{U}_{N}$. The
representation in \cite{budhirajadupuis00} applies to an infinite-dimensional Brownian motion, and thus strictly speaking the infimum
would be over a collection of controls indexed by $i\in\mathbb{N}$.
However, since those controls with $i>N$ have no effect on $\bar{\mu
}^{N}$ we can and will assume they are zero.

Based on \eqref{EqVariationalPrelimit}, the Laplace principle
will be established in two steps. First, in Section \ref
{SectLPLowerBound}, we establish the variational lower bound by
showing\vadjust{\goodbreak}
that for any sequence $(u^{N})_{N\in\mathbb{N}}$ with $u^{N}\in
\mathcal{U}_{N}$,
\begin{eqnarray} \label{EqVariationalLowerBound}
&& \liminf_{N\rightarrow\infty} \Biggl\{ \frac{1}{2}\Mean\Biggl[ \frac
{1}{N}\sum_{i=1}^{N}\int_{0}^{T}|u_{i}^{N}(t)|^{2}\,dt\Biggr] + \Mean
[ F(\bar{\mu}^{N})] \Biggr\} \nonumber\\[-8pt]\\[-8pt]
&&\qquad  \geq  \inf_{\Theta\in\mathcal{P}_{\infty}} \biggl\{\frac
{1}{2}\int_{\mathcal{R}}\int_{\mathbb{R}^{d_{1}}\times[0,T]} |y|^{2}
r(dy \times dt)\Theta_{\mathcal{R}}(dr) + F(\Theta_{\mathcal{X}})
\biggr\}.\nonumber
\end{eqnarray}
Second, in Section \ref{SectLPUpperBound}, we verify the variational
upper bound by showing that for any measure $\Theta\in\mathcal
{P}_{\infty}$ there is a sequence $(u^{N})_{N\in\mathbb{N}}$ with
$u^{N}\in\mathcal{U}_{N}$ such that
\begin{eqnarray} \label{EqVariationalUpperBound}
&& \limsup_{N\rightarrow\infty} \Biggl\{ \frac{1}{2}\Mean\Biggl[ \frac
{1}{N} \sum_{i=1}^{N}\int_{0}^{T}|u_{i}^{N}(t)|^{2}\,dt\Biggr] + \Mean
[ F(\bar{\mu}^{N})] \Biggr\}\nonumber\\[-8pt]\\[-8pt]
&&\qquad  \leq  \frac{1}{2}\int_{\mathcal{R}}\int_{\mathbb
{R}^{d_{1}}\times[0,T]} |y|^{2}  r(dy \times dt)\Theta_{\mathcal
{R}}(dr) + F(\Theta_{\mathcal{X}}).
\nonumber
\end{eqnarray}

To see that those two steps establish Theorem \ref{ThLaplacePrinciple},
first observe that
\begin{eqnarray*}
\inf_{\theta\in\mathcal{P}(\mathcal{X})}\biggl\{ F(\theta)+\inf
_{\Theta\in\mathcal{P}_{\infty}\dvtx\Theta_{\mathcal{X}}=\theta}\biggl\{
\frac{1}{2}\int_{\mathcal{R}}\int_{\mathbb{R}^{d_{1}}\times[0,T]}
|y|^{2} r(dy \times dt)\Theta_{\mathcal{R}}(dr)\biggr\} \biggr\} \\
= \inf_{\Theta\in\mathcal{P}_{\infty}}\biggl\{ \frac{1}{2}\int
_{\mathcal{R}}\int_{\mathbb{R}^{d_{1}}\times[
0,T]} |y|^{2} r(dy \times dt)\Theta_{\mathcal{R}}(dr)+F(\Theta
_{\mathcal{X}})\biggr\}.
\end{eqnarray*}
Hence, in view of \eqref{EqVariationalPrelimit}, we have to show that
for all $F\in\mathbf{C}_{b}(\mathcal{X})$,
\[
\inf_{u\in\mathcal{U}_{N}}J_{N}^{F}(u)\stackrel{N\rightarrow\infty
}{\longrightarrow}\inf_{\Theta\in\mathcal{P}_{\infty}} J_{\infty
}^{F}(\Theta),
\]
where
\begin{eqnarray*}
J_{N}^{F}(u)& \doteq &\frac{1}{2}\Mean\Biggl[ \frac{1}{N}\sum_{i=1}^{N}
\int_{0}^{T}|u_{i}(t)|^{2}\,dt\Biggr] + \Mean[ F(\bar{\mu
}^{N})],\\
J_{\infty}^{F}(\Theta) &\doteq& \frac{1}{2}\int_{\mathcal{R}}\int
_{\mathbb{R}^{d_{1}}\times[0,T]} |y|^{2} r(dy \times
dt)\Theta_{\mathcal{R}}(dr)+F(\Theta_{\mathcal{X}}).
\end{eqnarray*}
Let $\varepsilon> 0$. For the lower bound, choose $u^{N}\in\mathcal
{U}_{N}$, $N\in\mathbb{N}$, such that $J_{N}^{F}(u^{N})\leq\inf_{u\in
\mathcal{U}_{N}}J_{N}^{F}(u)+\varepsilon$. Then \eqref
{EqVariationalLowerBound} implies that
\[
\liminf_{N\rightarrow\infty}\inf_{u\in\mathcal
{U}_{N}}J_{N}^{F}(u)\geq\inf_{\Theta\in\mathcal{P}_{\infty}}
J_{\infty}^{F}(\Theta) - \varepsilon.
\]
For the upper bound, choose a probability measure $\Theta\in\mathcal
{P}_{\infty}$ such that $J_{\infty}^{F}(\Theta)\leq\inf_{\Theta\in
\mathcal{P}_{\infty}} J_{\infty}^{F}(\Theta)+\varepsilon$. Since $\inf
_{u\in\mathcal{U}_{N}} J_{N}^{F}(u)\leq J_{N}^{F}(\tilde{u})$ for any
$\tilde{u}\in\mathcal{U}_{N}$, \eqref{EqVariationalUpperBound} implies that
\[
\limsup_{N\rightarrow\infty}\inf_{u\in\mathcal
{U}_{N}}J_{N}^{F}(u)\leq\inf_{\Theta\in\mathcal{P}_{\infty}}
J_{\infty}^{F}(\Theta) + \varepsilon.
\]
Since $\varepsilon> 0$ is arbitrary, the assertion follows.\vadjust{\goodbreak}

There is a technical observation to be made about the probability
spaces and filtrations underlying the stochastic control problems,
namely that there is a certain flexibility in the choice of the the
stochastic bases. This flexibility will be needed in establishing the
variational upper bound. To be more precise we note that the
representation theorem in \cite{budhirajadupuis00} holds for any
stochastic basis rich enough to carry a sequence of independent
standard $(\tilde{\mathcal{F}}_{t})$-Wiener processes. The
filtration~$(\tilde{\mathcal{F}}_{t})$, which is assumed to satisfy the usual
conditions, need not be the filtration induced by the Wiener processes,
but may be strictly larger. As a consequence of assumption
\hyperlink{APrelimitSolutions}{(A3)}, the left-hand side of \eqref
{EqVariationalPrelimit} does not depend on the choice of the stochastic
basis. The stochastic optimal control problem on the right-hand side of
\eqref{EqVariationalPrelimit} can therefore be regarded in the weak
sense, that is, the infimum is taken over all suitable stochastic bases
(see Definition 4.2 in \cite{yongzhou99}, page 64). The definition of
the sets $\mathcal{U}_{N}$ and assumption \hyperlink{ATightness}{(A5)} are to be
understood accordingly.

As a consequence of the weak formulation of the control problems, in
the proof of the variational lower bound, the control processes
$u^{N}$, the driving Wiener processes $W^{1},\ldots,W^{N}$ and thus the
empirical measures $\bar{\mu}^{N}$ could live on stochastic bases which
vary with $N$. While we do not make this variation explicit, it is easy
to see that the arguments of Section \ref{SectLPLowerBound}, being weak
convergence arguments, do not rely on having a common filtered
probability space. The variational upper bound, on the other hand, will
be established in Section \ref{SectLPUpperBound} by taking an arbitrary
$\Theta\in\mathcal{P}_{\infty}$ and then constructing a sequence of
control processes and independent Wiener processes so that \eqref
{EqVariationalUpperBound} holds. The prelimit processes will be
coordinate processes on a common stochastic basis which, however, will
depend on the limit probability measure $\Theta$.


\section{Auxiliary constructions}
\label{SectAuxiliary}

This section collects useful results for characterizing those
probability measures in $\mathcal{P}(\mathcal{Z})$ which correspond to
a weak solution of \eqref{EqLimitControlSDE}. Let $\Theta\in\mathcal
{P}(\mathcal{Z})$. Recall from \eqref{ExQMarginal} the definition of
the mapping $\nu_{\Theta} \dvtx [0,T]\rightarrow\mathcal{P}(\mathbb
{R}^{d})$ induced by $\Theta$. The mapping $\nu_{\Theta}$ is
continuous. To check this, take any $t_{0} \in[0,T]$ and any sequence
$(t_{n}) \subset[0,T]$ such that $t_{n}\rightarrow t_{0}$. Then for
all $f\in\mathbf{C}_{b}(\mathbb{R}^{d})$, the fact that elements of
$\mathcal{X}$ are continuous and the bounded convergence theorem imply
\begin{eqnarray*}
\int_{\mathbb{R}^{d}}f(x) \nu_{\Theta}(t_{n})(dx) &=& \int_{\mathcal
{X}\times\mathcal{R}\times\mathcal{W}} f(\varphi(t_{n})) \Theta
(d\varphi \times dr \times dw) \\
&\stackrel{n\rightarrow\infty}{\longrightarrow}& \int_{\mathcal
{X}\times\mathcal{R}\times\mathcal{W}} f(\varphi(t_{0})) \Theta
(d\varphi \times dr \times dw) \\
&=& \int_{\mathbb{R}^{d}} f(x)\nu_{\Theta}(t_{0})(dx).
\end{eqnarray*}
Therefore $\nu_{\Theta}(t_{n})\rightarrow\nu_{\Theta}(t)$ in $\mathcal
{P}(\mathbb{R}^{d})$. The continuity of $\nu_{\Theta}$ implies that the
set $\{\nu_{\Theta}(t) \dvtx t \in[0,T]\}$ is compact in $\mathcal
{P}(\mathbb{R}^{d}) $.

The question of whether a probability measure $\Theta\in\mathcal
{P}(\mathcal{Z})$ corresponds to a weak solution of \eqref
{EqLimitControlSDE}. or, equivalently, of \eqref
{EqParamLimitControlSDE} with $\nu= \nu_{\Theta}$, can be conveniently
phrased in terms of an associated local martingale problem. We
summarize here the main facts that we will use (see
\cite{stroockvaradhan79}, \cite{kushner90}, Section~4.4, and
\cite{karatzasshreve91}, Section 5.4, e.g.).

Given $f\in\mathbf{C}^{2}(\mathbb{R}^{d}\times\mathbb{R}^{d_{1}})$,
define a real-valued process $(M_{f}^{\Theta}(t))_{t\in[0,T]}$
on the probability space $(\mathcal{Z},\mathcal{B}(\mathcal{Z}),\Theta
)$ by
\begin{eqnarray} \label{ExLimitMartingale}
M_{f}^{\Theta}(t,(\varphi,r,w))&\doteq& f(\varphi
(t),w(t))-f(\varphi(0),0) \nonumber\\[-8pt]\\[-8pt]
&&{}  -\int_{0}^{t}\int_{\mathbb{R}^{d_{1}}}\mathcal{A}_{s}^{\Theta
}(f)(\varphi(s),y,w(s))r_{s}(dy)\,ds,
\nonumber
\end{eqnarray}
where for $s\in[0,T]$, $x\in\mathbb{R}^{d}$, $y,z\in\mathbb
{R}^{d_{1}}$,
\begin{eqnarray} \label{ExLimitGenerator}
\mathcal{A}_{s}^{\Theta}(f)(x,y,z) &\doteq& \langle b(x,\nu
_{\Theta}(s))+\sigma(x,\nu_{\Theta}(s))y,\nabla
_{x}f(x,z)\rangle\nonumber\\
&&{}  +\frac{1}{2}\sum_{j,k=1}^{d}(\sigma{\sigma}^{\mathsf
{T}})_{jk}(x,\nu_{\Theta}(s))\,\frac{\partial^{2}f}{\partial
x_{j}\,\partial x_{k}}(x,z) \nonumber\\[-8pt]\\[-8pt]
&&{}  +\frac{1}{2}\sum_{l=1}^{d_{1}}\frac{\partial^{2}f}{\partial
z_{l}\,\partial z_{l}}(x,z) \nonumber\\
&&{}  +\sum_{k=1}^{d}\sum_{l=1}^{d_{1}}\sigma_{kl}(x,\nu
_{\Theta}(s))\,\frac{\partial^{2}f}{\partial x_{k}\,\partial z_{l}}(x,z).
\nonumber
\end{eqnarray}
The expression involving $\mathcal{A}_{s}^{\Theta}(f)$ in \eqref
{ExLimitMartingale} is integrated against time and the time derivative
measures $r_{s}$ of any relaxed control $r$. The measures $r_{s}$ are
actually not needed in that we may use $r(dy \times ds)$ in place of
$r_{s}(dy)\,ds$.

The key relation, which we formulate as a lemma, is a one-to-one
correspondence between weak solutions of \eqref
{EqLimitControlSDE} and a local martingale problem.

\begin{lemma} \label{LemmaMartingaleProblem}
Let $\Theta\in\mathcal{P}(\mathcal{Z})$ be such that \mbox{$\Theta(\{(\varphi
,r,w)\in\mathcal{Z}\dvtx w(0)=0\})=1$}. Then $\Theta$ corresponds to a weak
solution of \eqref{EqLimitControlSDE} if and only if
$M_{f}^{\Theta}$ is a~local martingale under $\Theta$ with respect to
the canonical filtration $(\mathcal{G}_{t})$ for all $f\in\mathbf
{C}^{2}(\mathbb{R}^{d}\times\mathbb{R}^{d_{1}})$.

Moreover, in order to show that $\Theta$ corresponds to a weak solution
of \eqref{EqLimitControlSDE}, it is enough to check the local
martingale property for those $M_{f}^{\Theta}$ where the test function
$f$ is a monomial of first or second order, that is, for the test functions
\begin{eqnarray*}
(x,z)&\mapsto& x_{k}, \qquad k\in\{1,\ldots,d\},\qquad (x,z)\mapsto
x_{j}x_{k},\qquad  j, k\in\{1,\ldots,d\},\\
(x,z)&\mapsto& z_{l}, \qquad l\in\{1,\ldots,d_{1}\},\qquad (x,z)\mapsto
z_{j}z_{l},\qquad  j, l\in\{1,\ldots,d_{1}\}, \\
(x,z)&\mapsto& x_{k}z_{l}, \qquad k\in\{1,\ldots,d\}, l\in\{1,\ldots
,d_{1}\}.
\end{eqnarray*}
\end{lemma}
\begin{pf}
See, for example, the proof of Proposition 5.4.6 in
\cite{karatzasshreve91}, page~315. Note that since the canonical process on the
sample space $(\mathcal{Z},\mathcal{B}(\mathcal{Z}))$ includes a
component which corresponds to the driving Wiener process, there is no
need to extend the probability space $(\mathcal{Z},\mathcal{B}(\mathcal
{Z}),\Theta)$ even if the diffusion coefficient $\sigma$ is degenerate.
\end{pf}
\begin{rmk}
There is a technical point here concerning the canonical filtration
$(\mathcal{G}_{t})$ in $\mathcal{B}(Z)$. That filtration is not
necessarily $\Theta$-complete or right-continuous, while in the
literature solutions to \mbox{SDEs} are usually defined with respect to
filtrations satisfying the usual conditions (i.e., containing all sets
contained in a set of measure zero and being right-continuous).
However, any stochastically continuous and uniformly bounded
real-valued process defined on some probability space $(\tilde{\Omega
},\tilde{\mathcal{F}},\tilde{\Prb})$ which is a martingale under~$\tilde
{\Prb}$ with respect to some filtration $(\tilde{\mathcal{F}}_{t})$, is
also a martingale under $\tilde{\Prb}$ with respect to $(\tilde{\mathcal
{F}}_{t+}^{\tilde{\Prb}})$, where $(\tilde{\mathcal{F}}_{t}^{\tilde{\Prb
}})$ denotes the $\tilde{\Prb}$-augmentation of $(\tilde{\mathcal
{F}}_{t})$ (see the solution to Exercise 5.4.13 in
\cite{karatzasshreve91}, page 392). The filtration $(\tilde{\mathcal
{F}}_{t+}^{\tilde{\Prb}})$ satisfies the usual conditions. Since the
localizing sequence of stopping times for a~local martingale can always
be chosen in such a way that the corresponding stopped processes are
bounded martingales, it follows that if $M_{f}^{\Theta}$ is a local
martingale under $\Theta$ with respect to $(\mathcal{G}_{t})$, then it
is also a local martingale under $\Theta$ with respect to $(\tilde
{\mathcal{G}}_{t+}^{\Theta})$. The local martingale property of the
processes~$M_{f}^{\Theta}$ under $\Theta$ with respect to the canonical
filtration $(\mathcal{G}_{t})$ thus implies that the canonical process
on $(\mathcal{Z},\mathcal{B}(\mathcal{Z}))$ solves \eqref
{EqLimitControlSDE} under $\Theta$ with respect to the filtration
$(\tilde{\mathcal{G}}_{t+}^{\Theta})$, which satisfies the usual conditions.
\end{rmk}
\begin{rmk}
The reason why we use a local martingale problem rather than the
corresponding martingale problem is that it gives more flexibility in
characterizing the convergence of It{\^{o}} processes which are not
necessarily of diffusion type. In Section \ref{SectExtensionsDelay},
we extend the Laplace principle of Theorem~\ref{ThLaplacePrinciple} to
interacting systems described by \mbox{SDEs} with delay. In that case,
the coefficients $b$, $\sigma$ are progressive functionals; thus, they
may depend on the entire trajectory of the solution process up to the
current time. An appropriate choice of the stopping times in the local
martingale problem gives control over the state process up to the
current time and not
only at the current time. In particular, the proof of Lemma \ref
{LemmaLimitMeasure} below, where the local martingale problem is used
to identify certain limit distributions, continues to work also for the
more general model of Section \ref{SectExtensionsDelay}.
\end{rmk}


\section{Variational lower bound}
\label{SectLPLowerBound}

In the proof of the lower bound (\ref{EqVariationalLowerBound}) we can
assume that
\begin{equation} \label{ExControlBound}
\Mean\Biggl[ \frac{1}{N}\sum_{i=1}^{N}\int
_{0}^{T}|u_{i}^{N}(t)|^{2}\,dt\Biggr] \leq  2\Vert F\Vert,
\end{equation}
since otherwise the desired inequality is automatic. Let $(u^{N})_{N\in
\mathbb{N}}$ be a sequence of control processes such that \eqref
{ExControlBound} holds. This implies in particular that for $\Prb
$-almost all $\omega\in\Omega$, all $N \in\mathbb{N}$, $i \in\{
1,\ldots,N\}$, $\int_{0}^{T}|u_{i}^{N}(t,\omega)|\,dt < \infty$.
Modifying the sequence $(u^{N})$ on a set of $\Prb$-measure zero has no
impact on the validity of \eqref{EqVariationalLowerBound}. Thus, we may
assume that $u^{N}_{i}(\cdot,\omega)$ has a finite first moment for all
$\omega\in\Omega$.

For each $N\in\mathbb{N}$, define a $\mathcal{P}(\mathcal{Z})$-valued
random variable by
\begin{equation} \label{ExFOMeasure}\quad
Q_{\omega}^{N}(B\times R\times D)\doteq  \frac{1}{N} \sum
_{i=1}^{N}\delta_{\bar{X}^{i,N}(\cdot,\omega)}(B)\cdot
\delta_{\rho_{\omega}^{i,N}}(R)\cdot\delta_{W^{i}(\cdot,\omega)}(D),
\end{equation}
$B\times R\times D\in\mathcal{B}(\mathcal{Z})$, $\omega\in\Omega$,
where $\bar{X}^{i,N}$ is the solution of \eqref
{EqPrelimitControlSDE} under $u^{N} = (u_{1}^{N},\ldots,u_{N}^{N})$,
and $\rho_{\omega}^{i,N}$ is the relaxed control induced by
$u_{i}^{N}(\cdot,\omega)$ according to~\eqref{ExControlRelaxation}. Notice
that $\rho_{\omega}^{i,N} \in\mathcal{R}_{1}$. The functional
occupation measures $Q^{N}$, $N\in\mathbb{N}$, just defined are
related to the Laplace principle by the fact that
\begin{eqnarray} \label{EqFOMeasureLaplace}
&& \frac{1}{2}\Mean\Biggl[ \frac{1}{N}\sum_{i=1}^{N} \int
_{0}^{T}|u_{i}^{N}(t)|^{2}\,dt\Biggr]  +  \Mean[ F(\bar{\mu
}^{N})] \nonumber\\
&&\qquad= \int_{\Omega}\biggl[ \int_{\mathcal{R}_{1}}\biggl( \frac{1}{2}\int
_{\mathbb{R}^{d_{1}}\times[0,T]} |y|^{2}  r(dy \times dt)\biggr)
Q_{\omega,\mathcal{R}}^{N}(dr)\\
&&\hspace*{193.1pt}{} + F(Q_{\omega,\mathcal{X}}^{N})\biggr]
\Prb(d\omega),
\nonumber
\end{eqnarray}
where $Q_{\omega,\mathcal{X}}^{N}$, $Q_{\omega,\mathcal{R}}^{N}$ denote
the first and second marginal of $Q_{\omega}^{N}\in\mathcal{P}(\mathcal
{Z})$, respectively, and we recall that $\mathcal{Z} = \mathcal
{X}\times\mathcal{R}_{1}\times\mathcal{W}$.

Thanks to assumption \hyperlink{ATightness}{(A5)} and the bound \eqref
{ExControlBound}, the first marginals of $(Q^{N})_{N\in\mathbb{N}}$ are
tight as random measures. The next lemma states that tightness of
$(Q^{N})_{N\in\mathbb{N}}$ as random measures follows. Thus we are
asserting tightness of the measures $\gamma^{N}\in\mathcal{P}(\mathcal
{P}(\mathcal{Z}))$ defined by $\gamma^{N}(A)=\Prb(Q^{N}\in A)$, $A \in
\mathcal{B}(\mathcal{P}(\mathcal{Z}))$.
\begin{lemma} \label{LemmaTightness}
The family $(Q^{N})_{N\in\mathbb{N}}$ of $\mathcal{P}(\mathcal
{Z})$-valued random variables is tight.
\end{lemma}
\begin{pf}
The first marginals of $(Q^{N})_{N\in\mathbb{N}}$ are tight by
assumption \hyperlink{ATightness}{(A5)} and~\eqref{ExControlBound}. Since the
third marginals are obviously tight, we need only prove tightness of
the second marginals. Observe that
\[
g(r)\doteq\int_{\mathbb{R}^{d_{1}}\times[0,T]} |y|^{2} r(dy \times
 dt)
\]
is a tightness function on $\mathcal{R}_{1}$, that is, it is bounded
from below and has compact level sets. To verify the last property
take\vadjust{\goodbreak}
$c\in(0,\infty)$ and let $R_{c}\doteq\{r\in\mathcal{R}_{1}\dvtx g(r)\leq
c\}$. By Chebyshev's inequality, for all $M > 0$,
{\renewcommand{\theequation}{$\ast$}
\begin{equation} \label{LemmaTightnessBound}
\sup_{r\in R_{c}}r(\{y\in\mathbb{R}^{d_{1}}\dvtx |y|>M\}\times
[0,T])\leq\frac{c}{M^{2}}.
\end{equation}}

\vspace*{-3pt}

\noindent Hence $R_{c}$ is tight and thus relatively compact as a subset of
$\mathcal{R}$. Consequently, any sequence in $R_{c}$ has a weakly
convergent subsequence with limit in $\mathcal{R}$. Let $(r_{n})\subset
R_{c}$ be such that $(r_{n})$ converges weakly to $r_{\ast}$ for some
$r_{\ast}\in\mathcal{R}$. It remains to show that $r_{\ast}$ has
finite first moment and that the first moments of $(r_{n})$ converge to
that of $r_{\ast}$. By H\"{o}lder's inequality and a version of Fatou's
lemma (cf. Theorem~A.3.12 in \cite{dupuisellis97}, page 307),
\[
\sqrt{T\cdot c}\geq\liminf_{n\rightarrow\infty}\int_{\mathbb
{R}^{d_{1}}\times[0,T]} |y| r_{n}(dy \times dt)\geq\int
_{\mathbb{R}^{d_{1}}\times[0,T]} |y| r_{\ast}(dy \times dt).
\]
Let $M>0$. By \eqref{LemmaTightnessBound} and H\"{o}lder's inequality
we have for all $r\in R_{c}$,
\[
\int_{\{y\in\mathbb{R}^{d_{1}}\dvtx |y|>M\}\times[0,T]} |y| r(dy \times
 dt)\leq\frac{c}{M}.
\]
Therefore, using weak convergence,
\begin{eqnarray*}
\limsup_{n\rightarrow\infty}\int_{\mathbb{R}^{d_{1}}\times[0,T]}
|y| r_{n}(dy \times dt)& \leq & \frac{c}{M}+\int_{\{y\in\mathbb
{R}^{d_{1}}\dvtx|y|\leq M\}\times[0,T]} |y| r_{\ast}(dy \times dt) \\
& \leq & \frac{c}{M}+\int_{\mathbb{R}^{d_{1}}\times[0,T]} |y|
r_{\ast}(dy \times dt).
\end{eqnarray*}
Since $M>0$ may be arbitrarily big, it follows that
\[
\lim_{n\rightarrow\infty}\int_{\mathbb{R}^{d_{1}}\times[0,T]} |y|
r_{n}(dy \times dt) = \int_{\mathbb{R}^{d_{1}}\times[0,T]} |y|
r_{\ast}(dy \times dt).
\]

We conclude that $g$ is a tightness function on $\mathcal{R}_{1}$. Now
define a function $G \dvtx\mathcal{P}(\mathcal{Z})\rightarrow[0,\infty]$ by
\[
G(\Theta)\doteq\int_{\mathcal{Z}}g(r) \Theta(d\varphi \times dr
\times dw).
\]
Then $G$ is a tightness function on second marginals in $\mathcal
{P}(\mathcal{Z})$ (see Theorem~A.3.17 in \cite{dupuisellis97}, page 309). Thus in order to prove tightness of the second
marginals of $(Q^{N})_{N\in\mathbb{N}}$ (as random measures) it is
enough to show that
\[
\sup_{N\in\mathbb{N}}  \Mean[ G(Q^{N})] < \infty.
\]
However, this follows directly from \eqref{ExControlBound}.
\end{pf}

In the next lemma we identify the limit points of $(Q^{N})$ as being
weak solutions of \eqref{EqLimitControlSDE} with probability
one. The proof is similar in spirit to that of Theorem 5.3.1 in \cite{kushner90}, page 102.
\begin{lemma} \label{LemmaLimitMeasure}
Let $(Q^{N_{j}})_{j\in\mathbb{N}}$ be a weakly convergent subsequence
of $(Q^{N})_{N\in\mathbb{N}}$. Let $Q$ be a $\mathcal{P}(\mathcal
{Z})$-valued random variable defined on some\vadjust{\goodbreak} probability space $(\tilde
{\Omega},\tilde{\mathcal{F}},\tilde{\Prb})$ such that $Q^{N_{j}}
\stackrel{j\to\infty}{\longrightarrow} Q$ in distribution.\vspace*{1pt} Then
$Q_{\omega}$ corresponds to a~weak solution of \eqref
{EqLimitControlSDE} for $\tilde{\Prb}$-almost all $\omega\in\tilde
{\Omega}$.
\end{lemma}
\begin{pf}
Set $I \doteq\{N_{j}, j\in\mathbb{N}\}$, and write $(Q^{n})_{n\in I}$
for $(Q^{N_{j}})_{j\in\mathbb{N}}$. By hypothesis, $Q^{n} \rightarrow
Q$ in distribution.

Recall from Lemma \ref{LemmaMartingaleProblem} in Section \ref
{SectAuxiliary} that a probability measure $\Theta\in\mathcal
{P}(\mathcal{Z})$ with $\Theta(\{(\varphi,r,w)\in\mathcal{Z} \dvtx w(0)=0\})
= 1$ corresponds to a weak solution of \eqref
{EqLimitControlSDE} if (and only if), for all $f \in\mathbf
{C}^{2}(\mathbb{R} ^{d}\times\mathbb{R}^{d_{1}})$, $M^{\Theta}_{f}$ is
a local martingale under $\Theta$ with respect to the canonical
filtration $(\mathcal{G}_{t})$, where $M^{\Theta}_{f}$ is defined by
\eqref{ExLimitMartingale}. Moreover, the local martingale property has
to be checked only for those~$M^{\Theta}_{f}$ where the test function
$f$ is a monomial of first or second order.

In verifying the local martingale property of $M^{\Theta}_{f}$ when
$\Theta= Q_{\omega}$ for some \mbox{$\omega\in\tilde{\Omega}$}, we will
work with randomized stopping times. Those stopping times live on an
extension $(\hat{\mathcal{Z}},\mathcal{B}(\hat{\mathcal{Z}}))$ of the
measurable space $(\mathcal{Z},\mathcal{B}(\mathcal{Z}))$ and are
adapted to a filtration $(\hat{\mathcal{G}}_{t})$ in $\mathcal{B}(\hat
{\mathcal{Z}})$, where
\[
\hat{\mathcal{Z}}\doteq\mathcal{Z}\times[0,1],\qquad \hat{\mathcal
{G}}_{t}\doteq\mathcal{G}_{t} \times\mathcal{B}([0,1]),\qquad  t \in[0,T],
\]
and $(\mathcal{G}_{t})$ is the canonical filtration in $\mathcal
{B}(\mathcal{Z})$. Any random object defined on $(\mathcal{Z},\mathcal
{B}(\mathcal{Z}))$ also lives on $(\hat{\mathcal{Z}},\mathcal{B}(\hat
{\mathcal{Z}}))$, and no notational distinction will be made.

Let $\lambda$ denote the uniform distribution on $\mathcal{B}([0,1])$.
Any probability measure $\Theta$ on $\mathcal{B}(\mathcal{Z})$ induces
a probability measure on $\mathcal{B}(\hat{\mathcal{Z}})$ given by $\hat
{\Theta}\doteq\Theta\times\lambda$. For each $k\in\mathbb{N}$,
define a stopping time $\tau_{k}$ on $(\hat{\mathcal{Z}},\mathcal
{B}(\hat{\mathcal{Z}}))$ with respect to the filtration $(\hat{\mathcal
{G}}_{t})$ by setting, for $(z,a)\in\mathcal{Z}\times[0,1]$,
\[
\tau_{k}(z,a)\doteq\inf\{ t\in[0,T]\dvtx v(z,t)\geq k+a\},
\]
where
\[
v((\varphi,r,w),t)\doteq\int_{\mathbb{R}^{d_{1}}\times[0,t]}
|y| r(dy \times ds)+\sup_{s\in[0,t]}|\varphi(s)|+\sup_{s\in[0,t]}|w(s)|.
\]
Note that the mapping $t\mapsto v((\varphi,r,w),t)$ is monotonic for all
$(\varphi,r,w)\in\mathcal{Z}$. Hence the stopping times have the
following properties. The boundedness of $\varphi$ and $w$ (being
continuous functions on a compact interval) and the boundedness of $\int
_{\mathbb{R}^{d_{1}}\times[0,T]} |y| r(dy \times ds)$ imply that
$\tau_{k}\nearrow T$ as $k\rightarrow\infty$ with probability one
under~$\hat{\Theta}$. The second property of note is that the mapping
\[
\mathcal{Z}\times[0,1]\ni(z,a)\quad\mapsto\quad\tau_{k}(z,a)\in[0,T]
\]
is continuous with probability one under $\hat{\Theta}$. To see this,
note that for every $z\in\mathcal{Z}$ the set
\[
A_{z}\doteq\{ c\in\mathbb{R}_{+}\dvtx v(z,s)=c\mbox{ for all }s\in
[t,t + \delta]\mbox{, some }t\in[0,T],\mbox{ some }\delta>0\}
\]
is at most countable. However, $\hat{z}\mapsto\tau_{k}(\hat{z})$
fails to be continuous at $(z,a)$ only when $k+a\in A_{z}$. Therefore,
by Fubini's theorem,
\begin{eqnarray*}
&&\hat{\Theta}\bigl(\{(z,a)\in\hat{Z}\dvtx\tau_{k}\mbox{ discontinuous at
}(z,a)\}\bigr)\\
&&\qquad = \int_{\hat{\mathcal{Z}}}\mathbf{1}_{A_{z}}(k + a)\hat
{\Theta}(dz \times da) \\
&&\qquad = \int_{\mathcal{Z}}\int_{[0,1]} \mathbf{1}_{A_{z}}(k +
a)\lambda(da)\Theta(dz) \\
&&\qquad = 0.
\end{eqnarray*}

Notice that if $M_{f}^{\Theta}$ is a local martingale with respect to
$(\hat{\mathcal{G}}_{t})$ under $\hat{\Theta}=\Theta\times\lambda$
with localizing sequence of stopping times $(\tau_{k})_{k\in\mathbb
{N}}$, then $M_{f}^{\Theta}$ is also a local martingale with respect to
$(\mathcal{G}_{t})$ under $\Theta$ with localizing sequence of stopping
times $(\tau_{k}(\cdot,0))_{k\in\mathbb{N}}$; see \hyperref[AppLocalMartingales]{Appendix}.
Thus it suffices to prove the martingale
property of $M_{f}^{\Theta}$ up till time $\tau_{k}$ with respect to
filtration $(\hat{\mathcal{G}}_{t})$ and probability measure $\hat
{\Theta}$.

Clearly, the process $M_{f}^{\Theta}(\cdot\wedge\tau_{k})$ is a $(\hat
{\mathcal{G}}_{t})$-martingale under $\hat{\Theta}$ if and only~if
\setcounter{equation}{3}
\begin{equation} \label{EqLemmaLimitMartingale}
\Mean_{\Theta\times\lambda}\bigl[ \Psi\cdot\bigl(M_{f}^{\Theta
}(t_{1}\wedge\tau_{k})-M_{f}^{\Theta}(t_{0}\wedge\tau_{k})
\bigr)\bigr] = 0
\end{equation}
for all $t_{0},t_{1}\in[0,T]$ with $t_{0}\leq t_{1}$, and $\hat
{\mathcal{G}}_{t_{0}}$-measurable $\Psi\in\mathbf{C}_{b}(\hat{\mathcal{Z}})$.

To verify the martingale property of $M_{f}^{\Theta}(\cdot\wedge\tau
_{k})$ it is enough to check that \eqref{EqLemmaLimitMartingale} holds
for any countable collection of times $t_{0}$, $t_{1}$ which is dense
in $[0,T]$ and any countable collection of functions $\Psi\in\mathbf
{C}_{b}(\hat{\mathcal{Z}})$ that generates the (countably many) $\sigma
$-algebras $\hat{\mathcal{G}}_{t_{0}}$. Recall that the collection of
test functions~$f$ for which a martingale property must be verified
consists of just monomials of degree one or two, and hence is finite.
Thus, there is a countable collection $\mathcal{T}\subset\mathbb
{N}\times[0,T]^{2}\times\mathbf{C}_{b}(\hat{\mathcal{Z}})\times
\mathbf{C}^{2}(\mathbb{R}^{d} \times \mathbb{R}^{d_{1}})$ of test
parameters such that if \eqref{EqLemmaLimitMartingale} holds for all
$(k,t_{0},t_{1},\Psi,f)\in\mathcal{T}$, then $\Theta$ corresponds to
a~weak solution of~\eqref{EqLimitControlSDE}.

Let $(k,t_{0},t_{1},\Psi,f)\in\mathcal{T}$. Define a mapping $\Phi=
\Phi_{(k,t_{0},t_{1},\Psi,f)}$ by
\[
\mathcal{P}(\mathcal{Z})\ni\Theta\quad\mapsto\quad\Phi(\Theta)\doteq\Mean
_{\Theta\times\lambda}\bigl[ \Psi\cdot\bigl(M_{f}^{\Theta
}(t_{1}\wedge\tau_{k})-M_{f}^{\Theta}(t_{0}\wedge\tau_{k})
\bigr)\bigr].
\]
We claim that the mapping $\Phi$ is continuous in the topology of weak
convergence on $\mathcal{P}(\mathcal{Z})$. To check this, take $\Theta
\in\mathcal{P}(\mathcal{Z})$ and any sequence $(\Theta_{l})_{l\in
\mathbb{N}}\subset\mathcal{P}(\mathcal{Z})$ that converges to $\Theta
$. Recall the definitions \eqref{ExLimitMartingale} and \eqref
{ExLimitGenerator}. As a consequence of assumption \hyperlink{ACoeffContinuity}{(A2)}
and by construction of the stopping time $\tau
_{k}$, the integrand in \eqref{EqLemmaLimitMartingale} is bounded;
thanks to assumption \hyperlink{ACoeffContinuity}{(A2)} and the almost sure
continuity of $\tau_{k}$, it is continuous with probability one under
$\hat{\Theta}\doteq\Theta\times\lambda$. By weak convergence and the
mapping theorem \cite{billingsley99}, page 21, it follows that
\begin{eqnarray} \label{EqLemmaLimitFirstPart}
&&\Mean_{\Theta_{l}\times\lambda}\bigl[ \Psi\cdot\bigl(M_{f}^{\Theta
}(t_{1}\wedge\tau_{k})-M_{f}^{\Theta}(t_{0}\wedge\tau_{k})
\bigr)\bigr] \nonumber\\[-8pt]\\[-8pt]
&&\qquad \stackrel{l\to\infty}{\longrightarrow}  \Mean_{\Theta\times
\lambda}\bigl[ \Psi\cdot\bigl(M_{f}^{\Theta}(t_{1}\wedge\tau
_{k})-M_{f}^{\Theta}(t_{0}\wedge\tau_{k})\bigr)\bigr].
\nonumber
\end{eqnarray}
Since the sequence $(\Theta_{l})$ converges to $\Theta$, the set $\{
\Theta_{l} \dvtx l\in\mathbb{N}\}\cup\{\Theta\}$ is compact in $\mathcal
{P}(\mathcal{Z})$. Recalling \eqref{ExQMarginal}, we find that the set
of probability measures $\{\nu_{\Theta_{l}}(t)\dvtx l\in\mathbb{N},t\in
[0,T]\}\cup\{\nu_{\Theta}(t)\dvtx t\in[0,T]\}$ has compact closure in
$\mathcal{P}(\mathbb{R}^{d})$. We claim that together with
assumption \hyperlink{ACoeffContinuity}{(A2)} and the construction of $\tau
_{k}$, this implies that
\[
\sup_{t\in[0,T],\hat{z}\in\hat{\mathcal{Z}}} \bigl|M_{f}^{\Theta
_{l}}\bigl(t\wedge\tau_{k}(\hat{z}),\hat{z}\bigr)-M_{f}^{\Theta}\bigl(t\wedge\tau
_{k}(\hat{z}),\hat{z}\bigr)\bigr| \stackrel{l\rightarrow\infty
}{\longrightarrow} 0.
\]
To see this, we consider, for example, the integral corresponding to
the first term in the drift, which is
\[
\int_{0}^{t\wedge\tau_{k}(\hat{z})}\langle b(\varphi(s),\nu_{\Theta
_{l}}(s)),\nabla_{x}f(\varphi(s),w(s))\rangle \,ds.
\]
By the assumed continuity properties of $b$ this converges uniformly in
$t\in[0,T],\hat{z}\in\hat{\mathcal{Z}}$ to
\[
\int_{0}^{t\wedge\tau_{k}(\hat{z})}\langle b(\varphi(s),\nu
_{\Theta}(s)),\nabla_{x}f(\varphi(s),w(s))\rangle \,ds,
\]
and a similar result holds for each of the other terms. Since $\Psi$ is
bounded, it follows that
\begin{eqnarray*}
&&\bigl| \Mean_{\Theta_{l}\times\lambda}\bigl[ \Psi\cdot
\bigl(M_{f}^{\Theta}(t_{1}\wedge\tau_{k})-M_{f}^{\Theta}(t_{0}\wedge
\tau_{k})\bigr)\bigr] \\
&&\qquad{} -\Mean_{\Theta_{l}\times\lambda}\bigl[ \Psi\cdot
\bigl(M_{f}^{\Theta_{l}}(t_{1}\wedge\tau_{k})-M_{f}^{\Theta
_{l}}(t_{0}\wedge\tau_{k})\bigr)\bigr] \bigr|  \stackrel
{l\rightarrow\infty}{\longrightarrow}  0.
\end{eqnarray*}
In combination with (\ref{EqLemmaLimitFirstPart}) this implies $\Phi
(\Theta_{l})\rightarrow\Phi(\Theta)$.

By hypothesis, the sequence $(Q^{n})_{n\in I}$ of $\mathcal{P}(\mathcal
{Z})$-valued random variables converges to $Q$ in distribution. Hence
the mapping theorem and the continuity of $\Phi$ imply that $\Phi
(Q^{n})\rightarrow\Phi(Q)$ in distribution.

Let $n\in I$. By construction of $Q^{n}$ and Fubini's theorem, for
$\omega\in\Omega$,
\begin{eqnarray*}
\Phi(Q_{\omega}^{n})& =&\Mean_{Q_{\omega}^{n}\times\lambda}
\bigl[\Psi\cdot\bigl(M_{f}^{Q_{\omega}^{n}}(t_{1}\wedge\tau_{k}) -
M_{f}^{Q_{\omega}^{n}}(t_{0}\wedge\tau_{k})\bigr)\bigr] \\
& =&
\frac{1}{n} \sum_{i=1}^{n} \int_{0}^{1}\Psi
\bigl((\bar{X}^{i,n}(\cdot,\omega),\rho^{i,n}_{\omega},W^{i}(\cdot,\omega)),a\bigr)
\\
&&\hspace*{38.7pt}{}\times\biggl(f\bigl(\bar{X}^{i,n}(t_{1}\wedge\bar{\tau}^{i,n}_{k},\omega
),W^{i}(t_{1}\wedge\bar{\tau}^{i,n}_{k},\omega)\bigr) \\
&&\hspace*{56.2pt}{} - f\bigl(\bar{X}^{i,n}(t_{0}\wedge\bar{\tau
}^{i,n}_{k},\omega),W^{i}(t_{0}\wedge\bar{\tau}^{i,n}_{k},\omega)
\bigr) \\
&&\hspace*{56.5pt}{} - \int_{t_{0}\wedge\bar{\tau
}^{i,n}_{k}}^{t_{1}\wedge\bar{\tau}^{i,n}_{k}}  \mathcal{A}^{\bar{\mu
}^{n}_{\omega}}_{s}(f)(\bar{X}^{i,n}(s,\omega),u^{n}_{i}(s,\omega
),\\
&&\hspace*{196.6pt}W^{i}(s,\omega))\,ds \biggr)\,da,
\end{eqnarray*}
where $\mathcal{A}^{\bar{\mu}_{\omega}^{n}}$ is defined according to
\eqref{ExLimitGenerator} with $\bar{\mu}_{\omega}^{n}$ in place of $\nu
_{\Theta}$, and $\bar{\tau}_{k}^{i,n}=\bar{\tau}_{k}^{i,n}(\omega,a)$
is defined like $\tau_{k}((\varphi,r,w),a)$ with $\varphi$ replaced by
$\bar{X}^{i,n}(\cdot,\omega)$, $r$ replaced by $\rho_{\omega}^{i,n}$,
the relaxed control corresponding to $u_{i}^{n}(\cdot,\omega)$, and $w$
replaced by $W^{i}(\cdot,\omega)$.

For all $a\in[0,1]$, by It\^{o}'s formula, it holds $\Prb
$-almost surely that
\begin{eqnarray*}
&&f\bigl(\bar{X}^{i,n}(t_{1}\wedge\bar{\tau
}_{k}^{i,n}),W^{i}(t_{1}\wedge\bar{\tau}_{k}^{i,n})\bigr)\\
&&\quad{} - f
\bigl(\bar{X}^{i,n}(t_{0}\wedge\bar{\tau}_{k}^{i,n}),W^{i}(t_{0}\wedge\bar
{\tau}_{k}^{i,n})\bigr) \\
&&\quad{}- \int_{t_{0}\wedge\bar{\tau}_{k}^{i,n}}^{t_{1}\wedge\bar{\tau
}_{k}^{i,n}} \mathcal{A}_{s}^{\bar{\mu}^{n}}(f)(\bar
{X}^{i,n}(s),u_{i}^{n}(s),W^{i}(s))\,ds \\
&&\qquad= \int_{t_{0}\wedge\bar{\tau
}_{k}^{i,n}}^{t_{1}\wedge\bar{\tau}_{k}^{i,n}} {\nabla_{x}f}^{\mathsf
{T}}(\bar{X}^{i,n}(s),W^{i}(s))\sigma(X^{i,n}(s),\bar{\mu
}^{n}(s))\,dW^{i}(s) \\
&&\qquad\quad{}+ \int_{t_{0}\wedge\bar{\tau}_{k}^{i,n}}^{t_{1}\wedge\bar{\tau
}_{k}^{i,n}} {\nabla_{z}f}^{\mathsf{T}}(%
\bar{X}^{i,n}(s),W^{i}(s))\,dW^{i}(s),
\end{eqnarray*}
where $\bar{\tau}_{k}^{i,n}=\bar{\tau}_{k}^{i,n}(\cdot,a)$ and $\bar{\tau
}_{k}^{i,n}$, $\bar{\mu}^{n}$, $\bar{X}^{i,n}$, $u_{i}^{n}$, are random
objects on $(\Omega,\mathcal{F})$.

By Fubini's theorem and Jensen's inequality, we have
\begin{eqnarray*}
&&\Mean[ \Phi(Q^{n})^{2}] \\
&&\qquad\leq\int_{0}^{1}\Mean\bigl[ \Mean_{Q_{\omega}^{n}}\bigl[ \Psi
(\cdot,a)\cdot\bigl(M_{f}^{Q_{\omega}^{n}}\bigl(t_{1}\wedge\tau_{k}(\cdot,a)\bigr)\\
&&\qquad\quad\hspace*{91pt}{} -
M_{f}^{Q_{\omega}^{n}}\bigl(t_{0}\wedge\tau_{k}(\cdot,a)\bigr)\bigr)
\bigr]^{2}\bigr]\, da.
\end{eqnarray*}
For all $a\in[0,1]$, by the It\^{o} isometry and because $\Psi
(\cdot,a)$ is $\mathcal{G}_{t_{0}}$-measurable, and $\tau_{k}(\cdot,a)$ is a
stopping time with respect to $(\mathcal{G}_{t})$, it holds that
\begin{eqnarray*}
&& \Mean\bigl[ \Mean_{Q_{\omega}^{n}}\bigl[ \Psi(\cdot,a)\cdot
\bigl(M_{f}^{Q_{\omega}^{n}}\bigl(t_{1}\wedge\tau_{k}(\cdot,a)\bigr) -
M_{f}^{Q_{\omega}^{n}}\bigl(t_{0}\wedge\tau_{k}(\cdot,a)\bigr)\bigr)
\bigr]^{2}\bigr] \\
&&\qquad  =
\Mean\bigl[\Mean_{Q_{\omega}^{n}}\bigl[\Psi
(\cdot,a)\cdot \mathbf{1}_{\{\tau_{k}(\cdot,a)\geq
t_{0}\}}\\
&&\qquad\quad\hspace*{7.7pt}{}\times\bigl(M_{f}^{Q_{\omega}^{n}}\bigl(t_{1}\wedge\tau_{k}(\cdot,a)
\bigr) - M_{f}^{Q_{\omega}^{n}}\bigl(t_{0}\wedge\tau_{k}(\cdot,a)\bigr)
\bigr)\bigr]^{2} \bigr]
\\
&&\qquad  =
\Mean\Biggl[ \Biggl( \frac{1}{n} \sum_{i=1}^{n} \int
_{t_{0}\wedge\bar{\tau}^{i,n}_{k}(\cdot,a)}^{t_{1}\wedge\bar{\tau
}^{i,n}_{k}(\cdot,a)}  \Psi(\cdot,a)\cdot\mathbf{1}_{\{\bar{\tau
}^{i,n}_{k}(\cdot,a)\geq t_{0}\}}\\
&&\qquad\quad\hspace*{95.7pt}{}\times\bigl({\nabla_{z}f}^\mathsf{T}
(\bar{X}^{i,n}(s),W^{i}(s)) \\
&&\qquad\hspace*{122.3pt}{}+{\nabla_{x}f}^\mathsf{T}(\bar{X}^{i,n}(s),W^{i}(s))\\
&&\qquad\quad\hspace*{131pt}{}\times\sigma
(X^{i,n}(s),\bar{\mu}^{n}(s))\bigr)\,dW^{i}(s) \Biggr)^{2} \Biggr]
\\
&&\qquad  =
\frac{1}{n^{2}} \sum_{i=1}^{n} \Mean\biggl[\int
_{t_{0}\wedge\bar{\tau}^{i,n}_{k}(\cdot,a)}^{t_{1}\wedge\bar{\tau
}^{i,n}_{k}(\cdot,a)}  \bigl|\Psi(\cdot,a)\cdot\mathbf{1}_{\{\bar{\tau
}^{i,n}_{k}(\cdot,a)\geq t_{0}\}}\\
&&\qquad\quad\hspace*{94.6pt}{}\times\bigl({\nabla_{z}f}^\mathsf{T}
(\bar{X}^{i,n}(s),W^{i}(s)) \\
&&\qquad\quad\hspace*{110pt}{}+ {\nabla_{x}f}^\mathsf{T}(\bar{X}^{i,n}(s),W^{i}(s))\\
&&\qquad\quad\hspace*{130.2pt}{}\times\sigma
(X^{i,n}(s),\bar{\mu}^{n}(s))\bigr)\bigr|^{2} \,ds \biggr]
\\
&&\qquad  \stackrel{n\rightarrow\infty}{\longrightarrow} 0.
\end{eqnarray*}

It follows that for each $(k,t_{0},t_{1},\Psi,f) \in\mathcal{T}$ there
is a set $Z_{(k,t_{0},t_{1},\Psi,f)} \in\tilde{\mathcal{F}}$ such that
$\tilde{\Prb}(Z_{(k,t_{0},t_{1},\Psi,f)}) = 0$ and
\[
\Phi_{(k,t_{0},t_{1},\Psi,f)}(Q_{\omega}) = 0 \qquad \mbox{for all }
\omega\in\tilde{\Omega} \setminus Z_{(k,t_{0},t_{1},\Psi,f)}.
\]
Let $Z$ be the union of all sets $Z_{(k,t_{0},t_{1},\Psi,f)}$,
$(k,t_{0},t_{1},\Psi,f) \in\mathcal{T}$. Since $\mathcal{T}$ is
countable, we have $Z \in\tilde{\mathcal{F}}$, $\tilde{\Prb}(Z) = 0$ and
\[
\Phi_{(k,t_{0},t_{1},\Psi,f)}(Q_{\omega}) = 0  \qquad\mbox{for all }
\omega\in\Omega\setminus Z,  (k,t_{0},t_{1},\Psi,f) \in\mathcal{T}.
\]
It follows that $Q_{\omega}$ corresponds to a weak solution of
\eqref{EqLimitControlSDE} for $\tilde{\Prb}$-almost all
$\omega\in\tilde{\Omega}$.
\end{pf}

The function $F$ in \eqref{EqVariationalLowerBound} is bounded and
continuous. The variational lower bound now follows from
\eqref{EqFOMeasureLaplace}, Lemmas \ref{LemmaTightness} and
\ref{LemmaLimitMeasure}, Fatou's lemma and the definition of~$I$.


\section{Variational upper bound}
\label{SectLPUpperBound}

Let $\Theta\in\mathcal{P}_{\infty}$. We will construct a sequence
$(u^{N})_{N\in\mathbb{N}}$ with $u^{N}\in\mathcal{U}_{N}$ on a common
stochastic basis such that (\ref{EqVariationalUpperBound}) holds
\begin{eqnarray*}
&& \limsup_{N\rightarrow\infty}\Biggl\{ \frac{1}{2}\Mean\Biggl[ \frac
{1}{N} \sum_{i=1}^{N}\int_{0}^{T}|u_{i}^{N}(t)|^{2}\,dt\Biggr] +  \Mean
[ F(\bar{\mu}^{N})] \Biggr\} \\
&&\qquad  \leq  \frac{1}{2}\int_{\mathcal{R}}\int_{\mathbb
{R}^{d_{1}}\times[0,T]} |y|^{2} r(dy \times dt)\Theta_{\mathcal
{R}}(dr)+F(\Theta_{\mathcal{X}}).
\end{eqnarray*}

Let $(\bar{X},\rho,W)$ be the canonical process on $\mathcal{Z}$ (cf.
end of Section \ref{SectModel}). Then $((\mathcal{Z},\mathcal
{B}(\mathcal{Z}),\Theta),(\tilde{\mathcal{G}}_{t+}^{\Theta}),(\bar
{X},\rho,W))$ is a weak solution of \eqref{EqLimitControlSDE}.
The filtration $(\tilde{\mathcal{G}}_{t+}^{\Theta})$ satisfies the
usual conditions, where $(\tilde{\mathcal{G}}_{t}^{\Theta})$ denotes
the $\Theta$-augmentation of the canonical filtration $(\mathcal
{G}_{t})$ (cf. Section \ref{SectAuxiliary}).

Since the relaxed control process $\rho$ appears linearly in
\eqref{EqLimitControlSDE}, it corresponds, as far as the
dynamics are concerned, to an ordinary $(\mathcal{G}_{t})$-adapted
process $u$, namely
\[
u(t,\omega)\doteq\int_{\mathbb{R}^{d_{1}}}y \rho_{\omega
,t}(dy), \qquad t\in[0,T], \omega\in\mathcal{Z},
\]
where $\rho_{\omega,t}$ is the derivative measure of $\rho_{\omega}$
at time $t$. For the associated costs, by Jensen's inequality,
\begin{eqnarray*}
\Mean\biggl[ \int_{0}^{T}|u(t)|^{2}\,dt\biggr] & = & \Mean\biggl[\int
_{0}^{T}\biggl\vert\int_{\mathbb{R}^{d_{1}}}y \rho
_{t}(dy)\biggr\vert^{2}\biggr] \\
& \leq & \Mean\biggl[ \int_{0}^{T}\int_{\mathbb{R}^{d_{1}}}|y|^{2}\rho
_{t}(dy)\biggr] \\
& = &\Mean\biggl[ \int_{\mathbb{R}^{d_{1}}\times[0,T]}|y|^{2}\rho(dy
\times dt)\biggr],
\end{eqnarray*}
whence $u$ performs at least as well as $\rho$. Let $\tilde{\rho}$ be
the relaxed control random variable corresponding to $u$ according to
\eqref{ExControlRelaxation}. In general, $\tilde{\rho}\neq\rho$.
However, since both $(\bar{X},\rho,W)$ and $(\bar{X},\tilde{\rho},W)$
are solutions of \eqref{EqLimitControlSDE} under $\Theta$ and
since the costs associated with $u$ and thus $\tilde{\rho}$ never
exceed the costs associated with $\rho$, we may and will assume that
$\rho=\tilde{\rho}$.

Define a probability space $(\Omega_{\infty},\mathcal{F}^{\infty},\Prb
_{\infty})$ together with a filtration $(\mathcal{F}_{t}^{\infty})$
as the countably infinite product of $(\mathcal{Z},\mathcal{B}(\mathcal
{Z}),\Theta)$ and $(\tilde{\mathcal{G}}_{t+}^{\Theta})$, respectively.
For a typical element of $\Omega_{\infty}$ let us write $\omega=
(\omega_{1},\omega_{2},\ldots)$. For $i\in\mathbb{N}$ define
\[
W^{i,\infty}(t,\omega)\doteq W(t,\omega_{i}),\qquad  u_{i}^{\infty
}(t,\omega_{i})\doteq u(t,\omega_{i}),\qquad \omega\in\Omega_{\infty
}, t\in[0,T].
\]
Let $\rho^{i,\infty}$ be the relaxed control random variable
corresponding to $u_{i}^{\infty}$. By construction, $(\rho^{i,\infty
},W^{i,\infty})$, $i\in\mathbb{N}$, are independent and identically
distributed with common distribution the same as that of $(\rho,W)$. In
particular, $W^{i,\infty}$, $i\in\mathbb{N}$, are independent
$d_{1}$-dimensional standard Wiener processes.

For $N \in\mathbb{N}$, let $\tilde{X}^{1,N},\ldots,\tilde{X}^{N,N}$ be
the solution to the system of \mbox{SDEs}
\begin{eqnarray*}
d\tilde{X}^{i,N}(t)& = & b(\tilde{X}^{i,N}(t),\tilde{\mu
}^{N}(t))\,dt + \sigma(\tilde{X}^{i,N}(t),\tilde{\mu
}^{N}(t))u^{\infty}_{i}(t)\,dt \\
&&{} + \sigma(\tilde{X}^{i,N}(t),\tilde{\mu}^{N}(t)
)\,dW^{i,\infty}(t),\qquad \tilde{X}^{i,N}(0) = x^{i,N},
\end{eqnarray*}
where $\tilde{\mu}^{N}(t)$ is the empirical measure of $\tilde
{X}^{1,N},\ldots,\tilde{X}^{N,N}$ at time $t$. Thus, $\tilde{X}^{i,N}$
solves \eqref{EqPrelimitControlSDE} with the same
deterministic initial condition as before, but on a different
stochastic basis.

For each $N\in\mathbb{N}$ define, in analogy with \eqref{ExFOMeasure},
a $\mathcal{P}(\mathcal{Z})$-valued random variable according to
\[
\tilde{Q}_{\omega}^{N}(B\times R\times D)\doteq  \frac{1}{N} \sum
_{i=1}^{N}\delta_{\tilde{X}^{i,N}(\cdot,\omega)}(B)\cdot\delta_{\rho
_{\omega}^{i,\infty}}(R)\cdot\delta_{W^{i,\infty}(\cdot,\omega)}(D),
\]
$B\times R\times D\in\mathcal{B}(\mathcal{Z})$, $\omega\in\Omega
_{\infty}$. In analogy with \eqref{EqFOMeasureLaplace} we have
\begin{eqnarray} \label{EqFOMeasureLaplace2}
&& \frac{1}{2}\Mean_{\infty}\Biggl[ \frac{1}{N}\sum_{i=1}^{N} \int
_{0}^{T}|u_{i}^{\infty}(t)|^{2}\,dt\Biggr] + \Mean_{\infty}[
F(\tilde{\mu}^{N})] \nonumber\\
&&\qquad = \int_{\Omega_{\infty}}\biggl[ \int_{\mathcal{R}_{1}}\biggl( \frac
{1}{2} \int_{\mathbb{R}^{d_{1}}\times[0,T]} |y|^{2}  r(dy \times
dt)\biggr) \tilde{Q}_{\omega,\mathcal{R}}^{N}(dr)\\
&&\qquad\quad\hspace*{167pt}{}+F(\tilde{Q}_{\omega
,\mathcal{X}}^{N})\biggr] \Prb_{\infty}(d\omega).
\nonumber
\end{eqnarray}
Since $(\tilde{\rho}^{i,\infty},W^{i,\infty})$, $i\in\mathbb{N}$, are
i.i.d., the second and third component of $(\tilde{Q}^{N})_{N\in\mathbb
{N}}$ are tight. Tightness of the first component is an immediate
consequence of assumption \hyperlink{ATightness}{(A5)}. Thus, $(\tilde
{Q}^{N})_{N\in\mathbb{N}}$ is tight as a family of $\mathcal
{P}(\mathcal{Z})$-valued random variables.

Let\vspace*{1pt} $\tilde{Q}$ be any limit point of $(\tilde{Q}^{N})_{N\in\mathbb
{N}}$ defined on some probability space $(\tilde{\Omega},\tilde{\mathcal
{F}},\tilde{\Prb})$. By Lemma \ref{LemmaLimitMeasure} and its proof, it
follows that, for $\tilde{\Prb}$-almost all $\omega\in\tilde{\Omega
}$, $\tilde{Q}_{\omega}$ corresponds to a weak solution of
\eqref{EqLimitControlSDE}. Moreover, since $(\rho^{i,\infty
},W^{i,\infty})$, $i\in\mathbb{N}$, are i.i.d. with common
distribution (under $\Prb_{\infty}$), the same as that of $(\rho,W)$
(under~$\Theta$), Varadarajan's theorem \cite{dudley02}, page 399,
implies that, for $\tilde{\Prb}$-almost all $\omega\in\tilde{\Omega}$,
\[
\tilde{Q}_{\omega|\mathcal{B}(\mathcal{R}_{1}\times\mathcal{W})} =
\Theta\circ(\rho,W)^{-1};
\]
that is, the joint distribution of the second and third component of
the canonical process on $\mathcal{Z}$ under a typical $\tilde
{Q}_{\omega}$ equals the joint distribution of the control and Wiener
process with which we started.

By assumption \hyperlink{ALimitSolution}{(A4)}, weak sense uniqueness holds for
\eqref{EqLimitControlSDE}. Therefore, for $\tilde{\Prb
}$-almost all $\omega\in\tilde{\Omega}$,
\[
\tilde{Q}_{\omega} = \Theta\circ(\bar{X},\rho,W)^{-1}.
\]
In view of \eqref{EqFOMeasureLaplace2}, the above
identification of the limit points establishes \eqref
{EqVariationalUpperBound}, the variational upper bound.


\section{Remarks and extensions}
\label{SectExtensions}

A feature of the weak convergence approach to large deviations is its
flexibility. To illustrate this point we show in Section~\ref
{SectExtensionsDelay} how to extend the Laplace principle established
in Theorem \ref{ThLaplacePrinciple} to weakly interacting systems
described by stochastic delay (or functional) differential equations.
Before, in Section \ref{SectExtensionsComparison}, we compare our
result to the classical large deviation principle (\mbox{LDP})
established in \cite{dawsongaertner87}.

\subsection{Comparison with existing results}
\label{SectExtensionsComparison}

In this subsection we compare our results with the now classical work
\cite{dawsongaertner87}. One of the main assumptions in the latter work
is the nondegeneracy of the diffusion coefficient $\sigma$. Although
the expression for the rate function is well-defined even if the
diffusion matrix~$\sigma\sigma^{T}$ is not invertible, the assumption
of nondegeneracy is important in the proof of the \mbox{LDP}.
Additionally, weak interaction is allowed only through the drift term.
Proofs proceed by first establishing a local version of the \mbox{LDP}
which is then lifted to a global result using careful exponential
probability estimates.

The approach taken in the current paper does not require any
exponential estimates and proofs cover the setting of a degenerate
$\sigma$ and models with weak interactions in both the drift and
diffusion coefficient. The significant additional assumption made in
the current work over \cite{dawsongaertner87} is \hyperlink{APrelimitSolutions}{(A3)};
we require strong existence and uniqueness of
solutions to \eqref{EqPrelimitSDE} whereas the cited paper
only assumes weak existence and uniqueness.

Of somewhat lesser significance is the difference in the topology
considered on $\mathcal{P}(\mathbb{R}^{d})$ and the space over which
the \mbox{LDP} is formulated. In particular, in~\cite{dawsongaertner87}
the drift coefficient $b$ need not be continuous on the entire product
space $\mathbb{R}^{d}\times\mathcal{P}(\mathbb{R}^{d})$, where
$\mathcal{P}(\mathbb{R}^{d})$ is equipped with the topology of weak
convergence, but only on $\mathbb{R}^{d}\times\mathcal{M}_{\infty}$,
where $\mathcal{M}_{\infty}$ is a set of probability measures on
$\mathcal{B}(\mathbb{R}^{d})$ which satisfy certain moment bounds in
terms of a~``Lyapunov function'' $\varphi
\dvtx\mathbb{R}^{d}\rightarrow\mathbb{R}$. The set $\mathcal{M}_{\infty
}$ is equipped with the ``inductive''
topology induced by $\varphi$ \cite{dawsongaertner87}, Section 5.1.
Additional assumptions in terms of this Lyapunov function are imposed
which, in particular, ensure that $(\mu^{N}(t))_{0\leq t\leq T}$ is
a~$\mathcal{M}_{\infty}$-valued process with continuous sample paths
(see (B.2)--(B.4) in \cite{dawsongaertner87}, Section 5.1). With some
additional work, we can relax assumption \hyperlink{ACoeffContinuity}{(A2)} on
the continuity of $b$, $\sigma$ in their second argument and, under
Lyapunov function conditions analogous to (B.2)--(B.4), obtain an \mbox
{LDP} in a space similar to the one used by \cite{dawsongaertner87},
namely $C([0,T],\mathcal{M}_{\infty})$. A~minor difficulty, with the
approach taken here, in working with $\mathcal{M}_{\infty}$ is that
the inductive topology is not metrizable. However, one can proceed as
follows. Let $\mathcal{P}_{\lambda}(\mathbb{R}^{d})$ be the set of all
probability measures $\nu\in\mathcal{P}(\mathbb{R}^{d})$ such that
$\int\lambda(x)\nu(dx)<\infty$, where $\lambda
(x)=|x|k_{0}(|x|,|x|)$ for some (suitable) symmetric, continuous,
nonnegative and nondecreasing function $k_{0}$ cf. \cite{rachev91},
page~123. The topology of $\lambda$-weak convergence, that is,
weak convergence plus convergence of $\lambda$-moments, makes $\mathcal
{P}_{\lambda}(\mathbb{R}^{d})$ a Polish space; cf. Theorems 6.3.1 and~6.3.3
in \cite{rachev91}, pages 130--134. Instead of \hyperlink{ACoeffContinuity}{(A2)}, we would assume that $b$, $\sigma$ are continuous
as functions defined on $\mathbb{R}^{d}\times\mathcal{P}_{\lambda
}(\mathbb{R}^{d})$ with $\mathcal{P}_{\lambda}(\mathbb{R}^{d})$
carrying the topology of $\lambda$-weak convergence. The function
$\lambda$ plays the role of the Lyapunov function $\varphi$ used in
\cite{dawsongaertner87}, Section 5.1. The only further modification would
regard assumption~\hyperlink{ATightness}{(A5)}. In addition to tightness of the
sequences of empirical measures $(\bar{\mu}^{N})$, one would have to
guarantee that the time marginals $\bar{\mu}^{N}(t)$ stay in $\mathcal
{P}_{\lambda}(\mathbb{R}^{d})$. An appropriate condition (which would
be analogous to conditions (B.2)--(B.4) in \cite{dawsongaertner87},
Section 5.1) could be formulated in terms of the
Lyapunov function.

The expression for the rate function given in equation (1.5) in
\cite{dawsongaertner87} is different from the one given in Theorem \ref
{ThLaplacePrinciple} of this paper. The integrand in particular
involves the maximization over a class of smooth test functions $f
\dvtx
\mathbb{R}^{d}\rightarrow\mathbb{R}$ at each time point $t\in[0,T]$.
In the case where the diffusion coefficient $\sigma$ is the identity
matrix, test functions $f(t,\cdot)$, $t\in[0,T]$, induce feedback controls
for \eqref{EqLimitControlSDE2} through $u(t,\omega) \doteq
\nabla_{x} f(t,\bar{X}(t,\omega))$, cf. Remark \ref{RmAltRate}. In
this way one can see, at least formally, the equivalence of our
expression for the rate function and the expression derived in
\cite{dawsongaertner87}.

\subsection{Processes with delay}
\label{SectExtensionsDelay}

Our approach allows one to treat more general It\^{o} equations than
those of diffusion type with very little additional effort. A good
example are \makebox{SDEs} whose coefficients are allowed to depend on
the entire past of the state trajectories. Let us make this more
precise. Suppose that the coefficients $b$, $\sigma$ are progressive
functionals defined on $[0,T]\times\mathcal{X}\times\mathcal
{P}(\mathbb{R}^{d})$, where we recall that $\mathcal{X}=\mathbf
{C}([0,T],\mathbb{R}^{d})$; that is, $b$, $\sigma$ are Borel measurable
and for each $t\in[0,T]$, $b$, $\sigma$ restricted to
$[0,t]\times\mathcal{X}\times\mathcal{P}(\mathbb{R}^{d})$ is
measurable with respect to $\mathcal{B}([0,t])\times\mathcal
{G}_{t}^{\mathcal{X}}\times\mathcal{B}(\mathcal{P}(\mathbb{R}^{d}))$
where $\mathcal{G}_{t}^{\mathcal{X}}$ is the $\sigma$-algebra
generated by the coordinate process on $\mathcal{X}$. Equation \eqref
{EqPrelimitSDE}, the prelimit equation for an individual particle (the
$i$th out of $N$), takes the form
\begin{equation} \label{EqPrelimitSFDE}
dX^{i,N}(t)=b(t,X^{i,N},\mu^{N}(t))\,dt+\sigma
(t,X^{i,N},\mu^{N}(t))\,dW^{i}(t).
\end{equation}
The system of $N$ equations given by \eqref{EqPrelimitSFDE} is a system
of stochastic functional differential equations or stochastic delay
differential equations (\mbox{SFDEs} or \mbox{SDDEs}). The
corresponding uncontrolled limit equation reads
\begin{equation} \label{EqLimitSFDE}
dX(t)=b(t,X,\Law(X(t)))\,dt+\sigma(t,X,\Law(X(t)))\,dW(t),
\end{equation}
while the controlled versions of \eqref{EqPrelimitSFDE} and \eqref
{EqLimitSFDE} will be
\begin{eqnarray}
\label{EqPrelimitControlSFDE2}
d\bar{X}^{i,N}(t) &=& b( t,\bar{X}^{i,N},\bar{\mu}^{N}(t)
)\,dt+\sigma( t,\bar{X}^{i,N},\bar{\mu}^{N}(t)) u_{i}(t)\,dt
\nonumber\\[-4pt]\\[-12pt]
&&{} + \sigma(t,\bar{X}^{i,N},\bar{\mu}^{N}(t)) u_{i}(t)\,dW^{i}(t),
\nonumber\\
\label{EqLimitControlSFDE}
d\bar{X}(t) &=& b( t,\bar{X},\operatorname{Law}(\bar{X}(t)))
\,dt+\biggl(\int_{\mathbb{R}^{d_{1}}}\sigma( t,\bar{X},\operatorname
{Law}(\bar{X}(t))) y\rho_{t}(dy)\biggr)\,dt \hspace*{-28pt}\nonumber\\[-8pt]\\[-8pt]
&&{} + \sigma( t,\bar{X},\operatorname{Law}(\bar{X}(t))) u(t)\,dW(t),
\nonumber
\end{eqnarray}
respectively. In \eqref{EqPrelimitControlSFDE2} $u_{i}$ is the
$i$th component of $u=(u_{1},\ldots,u_{N})$ for\break some~\mbox{$u\in\mathcal
{U}_{N}$}, while $\rho$ in \eqref{EqLimitControlSFDE} is an
adapted $\mathcal{R}_{1}$-valued random variable as in~\eqref
{EqLimitControlSDE}.\looseness=1

The Laplace principle can now be established in the same way as above
except for two points which need modification. Those are the
formulation of the local martingale problem in Section \ref
{SectAuxiliary} and the continuity
assumption~\hyperlink{APrelimitSolutions}{(A3$'$)}--\hyperlink{ATightness}{(A5$'$)}
the analogues of assumptions \hyperlink{APrelimitSolutions}{(A3)}--\hyperlink{ATightness}{(A5)}, which are obtained by replacing all
references to \eqref{EqPrelimitSDE}, \eqref{EqLimitSDE},
\eqref{EqPrelimitControlSDE}, \eqref{EqLimitControlSDE} with
\eqref{EqPrelimitSFDE}, \eqref{EqLimitSFDE}, (\ref
{EqPrelimitControlSFDE2}), (\ref{EqLimitControlSFDE}), respectively.

As to the martingale problem, we have to redefine the processes
$M^{\Theta}_{f}$ and the ``generators'' $\mathcal{A}_{s}^{\Theta}(f)$
according to
\begin{eqnarray*}
M_{f}^{\Theta}(t,(\varphi,r,w))&\doteq& f(\varphi(t),w(t)
) - f(\varphi(0),0) \\
&&{} -\int_{0}^{t}\int_{\mathbb{R}^{d_{1}}}\mathcal{A}_{s}^{\Theta
}(f)(\varphi,y,w(s))r_{s}(dy)\,ds,
\end{eqnarray*}
where for $s\in[0,T]$, $\varphi\in\mathcal{X}$, $y,z\in\mathbb{R}^{d_{1}}$,
\begin{eqnarray*}
\mathcal{A}_{s}^{\Theta}(f)(\varphi,y,z)&\doteq&\langle b(s,\varphi
,\nu_{\Theta}(s))+\sigma(s,\varphi,\nu_{\Theta}(s))y,\nabla
_{x}f(\varphi(s),z)\rangle\\
&&{} + \frac{1}{2} \sum_{j,k=1}^{d}(\sigma{\sigma}^{\mathsf
{T}})_{jk}(s,\varphi,\nu_{\Theta}(s))\,\frac{\partial
^{2}f}{\partial x_{j}\,\partial x_{k}}(\varphi(s),z) \\
&&{} + \frac{1}{2} \sum_{l=1}^{d_{1}}\frac{\partial^{2}f}{\partial
z_{l}\,\partial z_{l}}(\varphi(s),z) \\
&&{} + \sum_{k=1}^{d} \sum_{l=1}^{d_{1}} \sigma_{kl}(s,\varphi,\nu
_{\Theta}(s)) \,\frac{\partial^{2}f}{\partial x_{k}\,\partial
z_{l}}(\varphi(s),z).
\end{eqnarray*}
Notice that the test functions $f$ are still elements of $\mathbf
{C}^{2}(\mathbb{R}^{d}\times\mathbb{R}^{d_{1}})$. With these
redefinitions, Lemma \ref{LemmaMartingaleProblem} continues to hold.

Assumption \hyperlink{ACoeffContinuity}{(A2)} about the continuity of $b$,
$\sigma$ has to be modified in order to account for the time dependence
and be supplemented by a condition of uniform continuity and
boundedness, which is automatically satisfied in the diffusion case.

\begin{longlist}
\item[(A2$'$)] The functions $b(t,\cdot,\cdot)$,
$\sigma(t,\cdot,\cdot)$ are continuous, and uniformly continuous and bounded on
sets $B\times P$ whenever $B \subset\mathcal{X}$ is bounded and $P
\subset\mathcal{P}(\mathbb{R}^{d})$ is compact, uniformly in $t\in[0,T]$.
\end{longlist}

Define the set $\mathcal{P}_{\infty}^{\star}$ of probability measures
on $\mathcal{B}(\mathcal{Z})$ as the set $\mathcal{P}_{\infty}$ in
Section~\ref{SectLPLimit}, replacing reference to \eqref
{EqLimitControlSDE} with \eqref{EqLimitControlSFDE}. Then the
following large deviation (or Laplace) principle holds.
\begin{thrm}
\label{ThLaplacePrincipleDelay} Grant assumptions \textup{\hyperlink{AInitialCondition}{(A1)},
\hyperlink{ACoeffContinuity}{(A2$'$)}--\hyperlink{ATightness}{(A5$'$)}}.
Then the family of empirical measures $\{\mu
^{N},N\in\mathbb{N}\}$ associated with \eqref
{EqPrelimitSFDE} satisfies the Laplace principle with rate function
\[
\tilde{I}(\theta) = \inf_{\Theta\in\mathcal{P}_{\infty}^{\star}\dvtx\Theta
_{\mathcal{X}}=\theta}  \frac{1}{2}\int_{\mathcal{R}} \int_{\mathbb
{R}^{d_{1}}\times[0,T]} |y|^{2}  r(dy \times dt)\Theta_{\mathcal{R}}(dr).
\]
\end{thrm}

Note that there is also a simpler-looking form of the rate function as
in Remark~\ref{RmAltRate}. The proof of Theorem \ref
{ThLaplacePrincipleDelay} is completely analogous to that of
Theorem~\ref{ThLaplacePrinciple} given in Sections \ref
{SectLPLowerBound} and \ref{SectLPUpperBound}. The proof of Lemma \ref
{LemmaLimitMeasure}, in particular, and specifically the use of the
local martingale problem and randomized stopping times there was
tailored to fit not only the diffusion case, but the case of dynamics
with delay as well.

Finally, note that we could further generalize our model to include the
case of coefficients $b$, $\sigma$ which also depend on the past of the
empirical process. In this case, $b$, $\sigma$ would be progressive
functionals defined on $[0,T]\times\mathcal{X}\times\mathcal{P}(\mathcal
{X})$, and a Laplace principle could be established in the same way as before.

\begin{appendix}\label{AppLocalMartingales}

\section*{\texorpdfstring{Appendix: Local martingales with respect
to~\mbox{$(\hat{\mathcal{G}}_{t})$}~and~\mbox{$(\mathcal{G}_{t})$}}
{Appendix: Local martingales with respect to (Gt) and (Gt)}}

Let the notation be that of the proof of Lemma \ref{LemmaLimitMeasure}
in Section \ref{SectLPLowerBound}. Let $\Theta\in\mathcal{P}(Z)$, $f
\in\mathbf{C}^{2}(\mathbb{R}^{d})$, and set $M(t)\doteq M^{\Theta
}_{f}(t)$, $t \in[0,T]$. Notice that $M$ is a~random object defined on
$(\mathcal{Z},\mathcal{B}(\mathcal{Z}))$ with values in $\mathcal{X} =
\mathbf{C}([0,T],\mathbb{R}^{d})$, which can be identified with the
random object living on $(\hat{\mathcal{Z}},\mathcal{B}(\hat{\mathcal
{Z}}))$ given by
\[
\mathcal{Z}\times[0,1] \ni(z,s) \quad\mapsto\quad(M(t,z))_{t\in[0,T]} \in
\mathcal{X}.
\]
Let $k \in\mathbb{N}$. Suppose that $M(\cdot\wedge\tau_{k})$ is a
martingale under $\hat{\Theta} = \Theta\times\lambda$ with respect to
the canonical filtration $(\hat{\mathcal{G}}_{t})$ in $\mathcal{B}(\hat
{\mathcal{Z}})$. Set
\[
\tau^{\circ}_{k}(z) \doteq\tau_{k}(z,0),\qquad  z \in\mathcal{Z}.
\]
We claim that $M(\cdot\wedge\tau^{\circ}_{k})$ is a martingale under $\Theta
$ with respect to the canonical filtration $(\mathcal{G}_{t})$ in
$\mathcal{B}(\mathcal{Z})$.
\begin{pf*}{Proof of the martingale property}
Since $\tau_{k}$ is a $(\hat{\mathcal
{G}}_{t})$-stopping time and $\hat{\mathcal{G}}_{t} = \mathcal{G}_{t}
\times\mathcal{B}([0,1])$, $t \in[0,T]$, it follows that $\tau^{\circ
}_{k}$ is a $(\mathcal{G}_{t})$-stopping time. Moreover, $\tau^{\circ
}_{k}$ is also a $(\hat{\mathcal{G}}_{t})$-stopping time, because
$\mathcal{G}_{t}$ can be identified with $\mathcal{G}_{t}\times\{
\varnothing,[0,1]\}$, $t \in[0,T]$, and $(\mathcal{G}_{t}\times\{
\varnothing,[0,1]\})$ is a subfiltration of $(\hat{\mathcal{G}}_{t})$.

Let $s, t \in[0,T]$, $s \leq t$. We have to show that
\[
\Mean_{\Theta}[ M(t\wedge\tau^{\circ}_{k})\cdot\mathbf
{1}_{Z}] = \Mean_{\Theta}[ M(s\wedge\tau^{\circ}_{k})\cdot
\mathbf{1}_{Z}]  \qquad\mbox{for all } Z \in\mathcal{G}_{s}.
\]
Since $M(\cdot\wedge\tau_{k})$ is a martingale under $\hat{\Theta}$ with
respect to $(\hat{\mathcal{G}}_{t})$ and $\tau^{\circ}_{k}$ is also a
$(\hat{\mathcal{G}}_{t})$-stopping time, it follows that $M(\cdot\wedge\tau
_{k}\wedge\tau^{\circ}_{k})$ is a martingale under $\hat{\Theta}$ with
respect to $(\hat{\mathcal{G}}_{t})$. Yet for all $(z,t) \in\hat
{\mathcal{Z}}$,
\[
(\tau_{k}\wedge\tau^{\circ}_{k})(z,t) = \tau_{k}(z,t)\wedge\tau
_{k}(z,0) = \tau_{k}(z,0) = \tau^{\circ}_{k}(z)
\]
by construction of $\tau_{k}$ and definition of $\tau^{\circ}_{k}$.
Hence we know that
\[
\Mean_{\hat{\Theta}}[ M(t\wedge\tau^{\circ}_{k})\cdot\mathbf
{1}_{\hat{Z}}] = \Mean_{\hat{\Theta}}[ M(s\wedge\tau^{\circ
}_{k})\cdot\mathbf{1}_{\hat{Z}}]  \qquad\mbox{for all } \hat{Z}
\in\hat{\mathcal{G}}_{s}.
\]
Let $Z \in\mathcal{G}_{s}$. Then $Z\times[0,1] \in\hat{\mathcal
{G}}_{s}$ and, by Fubini's theorem,
\begin{eqnarray*}
\Mean_{\Theta}[ M(t\wedge\tau^{\circ}_{k})\cdot\mathbf
{1}_{Z}] &=& \int_{\mathcal{Z}} M\bigl(t\wedge\tau^{\circ}_{k}(z)\bigr)\cdot
\mathbf{1}_{Z}(z)  \Theta(dz) \\
&=& \int_{[0,1]} \int_{\mathcal{Z}} M\bigl(t\wedge\tau^{\circ
}_{k}(z)\bigr)\cdot\mathbf{1}_{Z\times[0,1]}(z,a)  \Theta(dz)\lambda(da) \\
&=& \int_{\mathcal{Z}\times[0,1]} M\bigl(t\wedge\tau^{\circ}_{k}(z)\bigr)\cdot
\mathbf{1}_{Z\times[0,1]}(z,a)  \hat{\Theta}(dz \times da) \\
&=& \Mean_{\hat{\Theta}}\bigl[ M(t\wedge\tau^{\circ}_{k})\cdot\mathbf
{1}_{Z\times[0,1]}\bigr] \\
&=& \Mean_{\hat{\Theta}}\bigl[ M(s\wedge\tau^{\circ}_{k})\cdot\mathbf
{1}_{Z\times[0,1]}\bigr] \\
&=& \int_{\mathcal{Z}\times[0,1]} M\bigl(s\wedge\tau^{\circ}_{k}(z)\bigr)\cdot
\mathbf{1}_{Z\times[0,1]}(z,a)  \hat{\Theta}(dz \times da)\\
&=& \Mean_{\Theta}[ M(s\wedge\tau^{\circ}_{k})\cdot\mathbf
{1}_{Z}].
\end{eqnarray*}
\upqed\end{pf*}
\end{appendix}

\section*{Acknowledgments}
The authors thank the Editor and an anonymous referee for their
critique and helpful comments and suggestions.

%

%
\printaddresses

\end{document}